\documentclass[sn-mathphys]{sn-jnl}

\jyear{2023}%

\theoremstyle{thmstyleone}%
\newtheorem{theorem}{Theorem}[section]

\newtheorem{teo}[theorem]{Theorem}
\newtheorem{obs}[theorem]{Remark}

\newtheorem{cor}[theorem]{Corollary}

\newtheorem{lemma}[theorem]{Lemma}

\newtheorem{prop}[theorem]{Proposition}

\theoremstyle{thmstyletwo}%

\theoremstyle{thmstylethree}%

\usepackage{lipsum}

\newcommand\blfootnote[1]{%
  \begingroup
  \renewcommand\thefootnote{}\footnote{#1}%
  \addtocounter{footnote}{-1}%
  \endgroup
}

\newcommand{\Rr}{\mathbb{R}}
\newcommand{\M}{\mathcal{M}}

\begin{document}

\title[Characterization of digraphs with three complementarity eigenvalues]{Characterization of digraphs with three complementarity eigenvalues}

\author[1]{\fnm{Diego} \sur{Bravo}}\email{dbravo27@gmail.com}

\author*[2]{\fnm{Florencia\blfootnote{Accepted in Journal of Algebraic Combinatorics.}} \sur{Cubría}}\email{fcubria@fing.edu.uy}

\author[2]{\fnm{Marcelo} \sur{Fiori}}\email{mfiori@fing.edu.uy}

\author[3]{\fnm{Vilmar} \sur{Trevisan}}\email{trevisan@mat.ufrgs.br}

\affil[1]{\orgdiv{R\&D Department}, \orgname{Tryolabs}, \state{Montevideo}, \country{Uruguay}}

\affil*[2]{\orgdiv{Instituto de Matemática y Estadística ``Rafael Laguardia''}, \orgname{Facultad de Ingenier\'{i}a, Universidad de la Republica}, \orgaddress{ \country{Uruguay}}}


\affil[3]{\orgdiv{Instituto de Matemática}, \orgname{Universidade Federal do Rio Grande do Sul}, \orgaddress{\city{Porto Alegre}, \country{Brazil}}}

\abstract{Given a digraph $D$, the complementarity spectrum of the digraph is defined as the set of complementarity eigenvalues of its adjacency matrix. This complementarity spectrum has been shown to be useful in several fields, particularly in spectral graph theory.
The differences between the properties of the complementarity spectrum for (undirected) graphs and for digraphs makes the study of the latter of particular interest, and characterizing strongly connected digraphs with a small number of complementarity eigenvalues is a nontrivial problem.
Recently, strongly connected digraphs with one and two complementarity eigenvalues have been completely characterized. In this paper we study {strongly connected} digraphs with exactly three elements in the complementarity spectrum, ending with a  complete characterization.
{This leads to a structural characterization of general digraphs having three complementarity eigenvalues.}
}

\maketitle

\section{Introduction}

Spectral Graph Theory studies the connection between structural graph properties and the spectral decomposition of certain matrices associated with the graph. {As prominent examples}, {from a powerful result of Sachs~\cite{Sachs1964},} the cycles of a given graph allow one to compute the characteristic polynomial of the adjacency matrix. On the other direction, {the seminal work of Fiedler~\cite{fiedler73} relates} the eigendecomposition of the Laplacian matrix and the connectivity of the graph for instance, and {it is well known that} certain properties of the eigenvalues and eigenvectors of the adjacency matrix give information about the automorphism group of the graph, or the regularity of the graph, just to name a few examples.%

One of the problems early addressed {in this area was the characterization of a graph by its spectrum. The initial belief that only isomorphic graphs share their spectra was soon proven wrong by the first examples of cospectral non-isomorphic graphs and digraphs from the 1950s \cite{Collatz1957}. Since then, the progress and the built knowledge in spectral graph theory is remarkable. We refer to the books \cite{Brouwer12,Doob,Cvetkovic1998} and references therein for an account of this progress.}


More recently, the concept of complementarity eigenvalues\footnote{Also referred as \textit{complementary eigenvalues}.} for matrices was introduced \cite{Seeger99}, and later on, applications were found in the context of graph theory \cite{Fernandes2017, Seeger2018}. It has been observed that this complementarity spectrum allows to distinguish more graphs than the traditional eigenvalues, so a natural question is whether a graph is determined by its complementarity eigenvalues {that is, whether only connected isomorphic graphs share their complementarity spectra.} This question remains unanswered as {of today.}

The concept of complementarity eigenvalues for graphs, and in particular its relationship to structural properties, was recently extended to digraphs~\cite{flor}. {It is worth noticing that it is shown that there exist} examples of non-isomorphic digraphs with the same complementarity spectrum. However, several questions concerning digraph characterization through the complementarity spectrum remain open. {In particular, determining which digraphs have a small number of complementarity eigenvalues is a proposed problem.}

In this paper we address the problem of characterizing all the strongly connected digraphs with exactly three complementarity eigenvalues, {which will be denoted $\mathcal{SCD}_{3}$}.	
For the traditional spectrum, the characterization of graphs with few eigenvalues seems to be easier than for digraphs \cite{Olivieri,Doob}. In the context of the complementarity spectrum, the problem of describing all graphs with three complementarity eigenvalues presents no challenges, since necessarily these graphs must have three vertices or less \cite{Seeger2018}, and therefore the problem becomes trivial. However, this characterization problem becomes interesting for digraphs, since this size restriction is no longer present.

In \cite{flor}, all digraphs with one or two complementarity eigenvalues are {completely characterized,} and the characterization problem for digraphs with three complementarity eigenvalues is posed. In this work we extend this {characterization} for these digraphs.

Our main results are as follows. We first observe that in the context of digraphs we may assume that digraphs are strongly connected (see Section~\ref{sec:preliminares}). Our preliminary result, Theorem~\ref{caracterizacion}, gives a characterization of {$\mathcal{SCD}_3$}
in terms of their cyclic structure. We then use this structural characterization to prove the main result of this paper, which is a full characterization theorem in Section \ref{sec:teorema_todo_junto}. In other words, we determine exactly which strongly connected digraphs have this structural property, thus giving a complete characterization of {$\mathcal{SCD}_3$}.

We build the proof of this result through the paper, structuring it as follows. In Section \ref{sec:preliminares} we present the preliminaries and useful results. In particular, we present the definition of complementary eigenvalues for matrices, for digraphs, and the main results connecting these eigenvalues with structural properties of the digraph. {In Section~\ref{sec:familes}, we present families of 
digraphs in {$\mathcal{SCD}_3$}.
In particular, in Subsection \ref{sec:basic_families} we present two fundamental examples 
namely, the $\infty$-digraphs, and the $\theta$-digraph. In sections \ref{sec:digrafos_con_infinito} and \ref{sec:digrafos_con_teta}, we see how these digraphs may be modified, by adding arcs, while keeping three complementarity eigenvalues. We are able to list seven of these {types of digraphs}, three of them having an underlying $\infty$-subdigraphs and four of them having an underlying $\theta$-subdigraph} {and not an underlying $\infty$-subdigraph}.

It turns out that these {seven types} are the only digraphs in {$\mathcal{SCD}_3$}
Section~\ref{sec:char} contains the proof of this result. Giving the cyclic structure characterization we give in Theorem~\ref{caracterizacion}, the problem has a natural combinatorial flavor and so does our approach to prove the result, which is algebraic in nature. 

\section{Preliminaries}
\label{sec:preliminares}

Let $D=(V,E)$ be a finite simple digraph  with vertices labelled as $1, \hdots, n$. The adjacency matrix of $D$ is defined as $A(D)=(a_{ij})$ where
\[a_{ij}=
\begin{cases}
1 \quad \text{if }(i,j)\in E,\\
0 \quad \text{otherwise.}
\end{cases}
\]

The multiset of roots of the characteristic polynomial of $A(D)$, counted with their multiplicities, is the \textit{spectrum} of $D$, {denoted by $Sp(D)$}. If $D$ is a digraph with strongly connected components $D_1, \hdots,D_k$, then $Sp(D)=\sqcup_{i=1}^k Sp(D_i)$ where $\sqcup$ denotes the union of multisets.

Throughout this paper, $\rho(\cdot)$ denotes the spectral radius (i.e., the largest module of the eigenvalues) of a matrix. For nonnegative irreducible matrices, it is well known that the spectral radius coincides with the largest real eigenvalue, due to the Perron-Frobenius Theorem. Additionally, in this case, the spectral radius is simple and may be associated with an eigenvector $x>\textbf{o}$, where \textbf{o} denotes the null vector in $\Rr^n$ and $\geq$ {(or $>$)} means that the inequality holds for every coordinate.

This real positive value $\rho(A(D))$ is called the \emph{spectral radius of the digraph $D$} and is denoted by $\rho(D)$. As we will see, the spectral radius of the digraph plays a fundamental role in the results obtained in this paper.

A digraph $H=(V',E')$ is a subdigraph of $D$ (denoted $H\leq D$) if $V'\subseteq V$ and $E'\subseteq E$. We say that $H$ is an induced subdigraph if $E'=E \cap (V'\times V')$ and a proper subdigraph if $E'\neq E$.\\

The following Lemma is a consequence of Perron-Frobenius Theorem, but we state it here for easy reference.
\begin{lemma}\label{lem:sub} Let $H$ be a proper subdigraph of a strongly connected digraph $D$. Then $\rho(H) < \rho(D)$.
\end{lemma}

We use the term \textit{cycle} to refer to a directed cycle in a digraph. As noted in \cite{flor}, the spectral radius of the cycle digraph $\vec{C}_n$, is $\rho(\vec{C}_n)=1.$\\

The Eigenvalue Complementarity Problem (EiCP) introduced in \cite{Seeger99} has found many applications in different fields of science, engineering and economics \cite{Adly2015,Facchinei2007,Pinto2008,Pinto2004}.

Given a matrix $A \in \M_n(\mathbb{R})$, the set of \textit{complementarity eigenvalues} is defined as those $\lambda \in \Rr$ such that there exists a vector $x \in \Rr^n$, not null and nonnegative, verifying
 $Ax \geq \lambda x$, and
\begin{equation*}
\label{complement}
\langle  x, Ax-\lambda x \rangle =0.
\end{equation*}

If we write $w=Ax-\lambda x\geq \textbf{o}$, the previous condition results in\[x^tw=0 \]
which means to ask for
\[x_i=0 \quad \text{or} \quad w_i=0 \text{ for all } i=1 \hdots, n.\]

This last condition is called \textit{complementarity condition}.

The set of all complementarity eigenvalues of a matrix $A$ is called the \textit{complementarity spectrum} of $A$, and it is denoted $\Pi(A)$. Unlike the regular spectrum of a matrix, the complementarity spectrum is a set (not a multiset), and the number of complementarity eigenvalues is not determined by the size of the matrix.

It is known that if $\lambda$ is a complementarity eigenvalue of $A$, then it is a complementarity eigenvalue of $PAP^t$ as well, for every permutation matrix $P$ \cite{Pinto2008}.

This fact allows us to define the complementarity spectrum of a digraph, since the complementarity spectrum is invariant in the family of adjacency matrices associated to the digraph.

The following theorem \cite{flor}, which extends an existing result for graphs, leads to a simple and useful characterization of the complementarity eigenvalues for digraphs. It allows us to characterize the complementarity eigenvalues of a digraph in terms of its structural properties.

\begin{teo}
\label{compl_spect_induced_subdigraphs} Let $D$ be a digraph and $\Pi(D)$ its complementarity spectrum. Then
\[\Pi(D)=\{\rho(H): {H \text{ induced strongly connected subdigraph of }D}\}.\]
\end{teo}

Particularly, the complementarity spectrum of a digraph is then a set of nonnegative real numbers containing $0$.\\


{
The next result, which can be found in~\cite{flor}, allows one to focus on the study of strongly connected digraphs.
\begin{prop}\label{prop:strong}
Let $D$ be a digraph and $D_1, \hdots, D_k$ the digraph generated by the strongly connected components of $D$. Then,
\[
\Pi(D)=\cup_{i=1}^k \Pi(D_i).
\]
\end{prop}}

The following result, also appearing in \cite{flor}, relates the complementarity spectrum with the cyclic structure of the digraph, showing that  the complementarity spectrum encodes structural properties of the digraph.

\begin{theorem}\label{tm:char} Let $D$ be a digraph and $\Pi(D)$ its complementarity spectrum. The three statements in (1) are equivalent to each other, and the three statements in (2) are equivalent to each other.
\begin{enumerate}
\item
\begin{enumerate}
    \item[(a)] $\Pi(D)=\{0\}$,
    \item[(b)] $\#\Pi(D)=1$,
    \item[(c)] $D$ is acyclic.
    \end{enumerate}
\item
    \begin{enumerate}
    \item[(a)] $\Pi(D)=\{0,1\}$,
    \item[(b)] $\#\Pi(D)=2$,
    \item[(c)] $D$ is not acyclic and its strongly connected components are either cycles or isolated vertices.
    \end{enumerate}
\end{enumerate}
\end{theorem}

Here, $\#\Pi(D)$ denotes the cardinality of $\Pi(D)$. {The result of Theorem~\ref{tm:char} shows that the }  digraphs with {one or two} complementarity {eigenvalues} were completely {determined in \cite{flor}}. 
From Theorem  \ref{tm:char} we conclude that 0 and 1 are always complementarity eigenvalues of any digraphs that has a cycle as a subdigraph.

{The set of all strongly connected digraphs with exactly $t$ elements in their complementarity spectrum will be denoted $\mathcal{SCD}_t$. It is easy to see that  $\mathcal{SCD}_1$ contains only one digraph consisting of an isolated vertex and $\mathcal{SCD}_2$ contains only cycles (of any size).}

{In the following sections we will precisely describe the types of digraphs in $\mathcal{SCD}_3$. To start with, we present a structural characterization of them}.
\begin{theorem}
\label{caracterizacion}
{$\mathcal{SCD}_3$ is the set of strongly connected digraphs whose only proper induced strongly connected subdigraphs are isolated vertices and cycles.}

\end{theorem}
\begin{proof}
Let $D$ be a digraph in $\mathcal{SCD}_3$. We show that all induced strongly connected proper subdigraphs of $D$ are either isolated vertices or cycles. We first notice that, by Theorem~\ref{tm:char}(1), $D$ needs to have a cycle as a subdigraph, hence $1 \in \Pi(D)$. Now let $D'$ be an {induced strongly connected} proper subdigraph of $D$; If $D'$ is neither a cycle nor an isolated vertex then, using Perron-Frobenius Theorem we have that \[1<\rho(D')<\rho(D),\] and therefore $\#\Pi(D)\geq 4$ which contradicts the cardinality of the complementarity spectrum.\\
Reciprocally, let $D$ be a strongly connected digraph whose only proper induced strongly connected subdigraphs are cycles and isolated vertices. Then we observe that $0,1 \in\Pi(D)$. {Given that $D$ is not a cycle, by Lemma \ref{lem:sub} we have that $1<\rho(D)$, hence $\Pi(D)=\{0,1, \rho(D)\}$ and $D$ belongs to $\mathcal{SCD}_3$.}\\
\end{proof}
The remainder of the paper is devoted to detect all the digraphs in $\mathcal{SCD}_3$.


\section{Families of digraphs in $\mathcal{SCD}_3$}\label{sec:familes}


Families in Section 3.1 and their digraphs are called basic, since, as it will be proved later on, basic digraphs are subdigraphs of any strongly connected noncycle digraph, and all digraphs in $\mathcal{SCD}_3$ are obtained by suitably adding arcs to a basic
digraph.

\subsection{Basic Families}\label{sec:basic_families}

We  first present our basic {types} of 
digraphs 
in {$\mathcal{SCD}_3$} which we call $\infty$-digraph and $\theta$-digraph.\\

\noindent \textbf{$\infty$-digraph (coalescence of cycles)}

The $\infty$-digraph $\infty=\infty(r,s)$ is the coalescence of two cycles $\vec{C_r}$ and $\vec{C_s}$.
It is easy to see that the only strongly connected induced subdigraphs are the cycles $\vec{C_r}$ and $\vec{C_s}$, besides the digraph $\infty$ itself and isolated vertices. Therefore, by virtue of Theorem \ref{compl_spect_induced_subdigraphs}, the complementarity spectrum can be computed by means of the spectral radii of these induced subdigraphs:
\[\Pi(\infty)=\{0,1,\rho(\infty)\}.\]
Figure \ref{fig:infinito} shows one example of a digraph in this family and a schematic representation.

\begin{figure}[h!]
\centering
\includegraphics[scale=0.18]{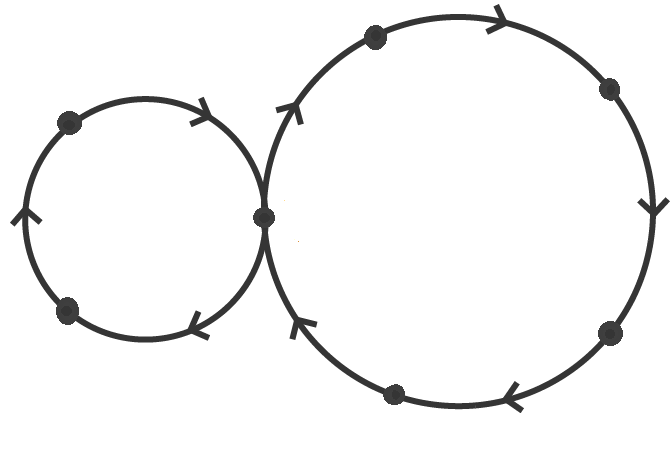}
\hspace{0.5cm}
\includegraphics[scale=0.18]{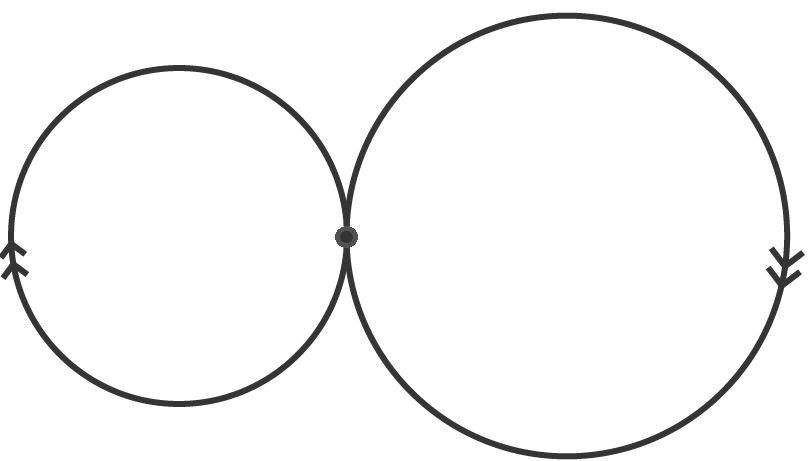}
\caption{Digraph $\infty(3,5)$ and a schematic representation of the same digraph.}
\label{fig:infinito}
\end{figure}

In the figures throughout this manuscript, we will represent digraphs using the following convention: a single arrow between two vertices indicates one arc joining them, while a double arrow indicates that there may be other vertices in the path joining them.\\


\noindent \textbf{$\theta$-digraph.}

The $\theta$-digraph \cite{Lin2012,flor}  consists of three directed paths $\vec{P}_{a+2}, \vec{P}_{b+2}, \vec{P}_{c+2}$ such that the initial vertex of $\vec{P}_{a+2}$ and $\vec{P}_{b+2}$ is the terminal vertex of $\vec{P}_{c+2}$, and the initial vertex of  $\vec{P}_{c+2}$ is the terminal vertex of $\vec{P}_{a+2}$ and $\vec{P}_{b+2}$, as shown in Figure \ref{fig:prohibido}. It will be denoted by $\theta(a, b, c)$ or simply by $\theta$.
\begin{figure}[h!]
\centering
\includegraphics[scale=0.15]{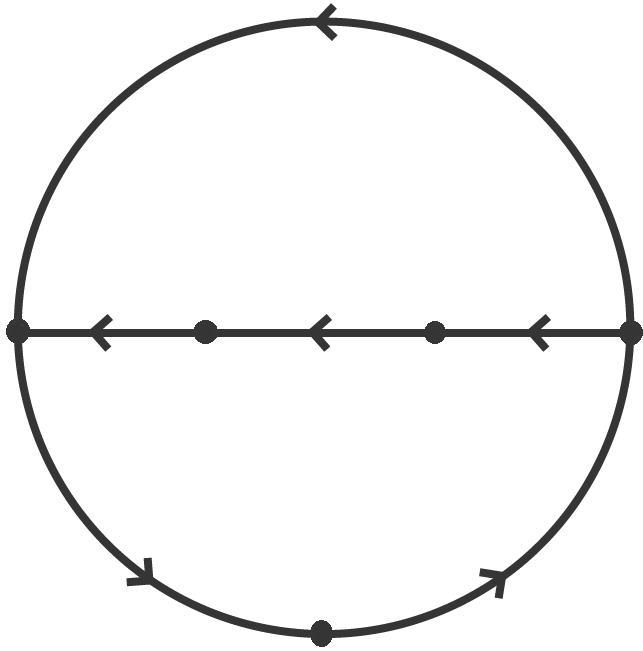}
\hspace{1cm}
\includegraphics[scale=0.15]{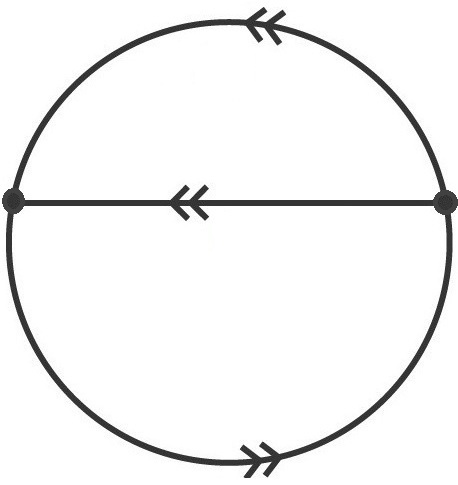}
\caption{Digraph $\theta(0,2,1)$ and a schematic representation of the same digraph.}
\label{fig:prohibido}
\end{figure}

Since ${\theta}(a,b,c)$ and ${\theta}(b,a,c)$ are isomorphic and we are not considering digraphs with multiple arcs, we can assume $a \leq b$ and $b > 0$,  without loss of generality.

The digraph $\theta(a,b,c)$ has $n=a+b+c+2$ vertices an the only strongly connected induced subdigraphs are the cycles $\vec{C_r}$ and $\vec{C_s}$ (where $r=a+c+2$ and $s=b+c+2$), in addition to the digraph $\theta$ itself, and isolated vertices. Therefore, we have,
\[\Pi({\theta})=\{0,1,\rho({\theta})\}.\]

{Note that an equivalent construction of the $\theta$-digraph can be made by taking a cycle, and adding a simple path joining two different vertices.}

These two {types}: $\infty$ and $\theta$ digraphs, are the only strongly connected bicyclic digraphs \cite{Lin2012}. Every  {digraph in $\mathcal{SCD}_{t}$ with $t\geq 3$ contains one of these two types of digraphs as a subdigraph, as the next Proposition shows}.

\begin{prop} \label{subdigrafosInftyTheta}
If $D$ is a strongly connected digraph different from a cycle {and an isolated vertex  (i.e. $D$ belongs to $\mathcal{SCD}_t$ for some $t\geq 3$)}, then it has an $\infty$-subdigraph or a $\theta$-subdigraph.
\end{prop}
\begin{proof}
Since $D$ is strongly connected, it has a cycle $\vec{C_r}$ as a subdigraph, for some integer $r$. Since $D$ is not a cycle, there are vertices $x,y$ in this cycle and a non-trivial path from $x$ to $y$, and not all the arcs belongs to $\vec{C_r}$. If $x=y$, then $D$ has an infinity digraph as a subdigraph. If $x \neq y$ then $D$ has a $\theta$ digraph as a subdigraph. The fact that $\#\Pi(D) > 1$ follows from item (1) of Theorem~\ref{tm:char}.
\end{proof}

\subsection{Digraphs with an $\infty$-subdigraph}
\label{sec:digrafos_con_infinito}

Let us now study in which ways we can modify {the $\infty$-digraph}, maintaining the number of complementarity eigenvalues. We first present two examples of digraphs in $ \mathcal{SCD}_3$, containing the $\infty$-digraph as a subdigraph. {In the next section we show that they are the only ones with this property.}

For simplicity, we will refer to vertices in $\vec{C}_r$ as $1,2, \hdots, r$, and to vertices in $\vec{C}_s$ as $1',2', \hdots, s'$, identifying $1$ with $1'$.\\

\textbf{Type 1 digraph}

Consider the digraph $D_1(r,s)=\infty(r,s)\cup\{e\}$, where $e$ denotes the arc connecting the vertex $r$ with the vertex $2'$, as shown in Figure \ref{fig_tipo1and2} (left). The only strongly connected induced subdigraphs of $D_1$ are both cycles $\vec{C_r}$ and $\vec{C_s}$, besides the digraph $D_1$ itself and isolated vertices. Then, we have
\[\Pi(D_1)=\{0,1,\rho(D_1)\}.\]

Observe that in this case the added arc distinguishes both cycles, and therefore we cannot assume $r\leq s$ {like in the previous example}. In other words, it is not the same to add an arc from the smaller cycle to the larger one, than the other way around.

\begin{figure}[h!]
\centering
\includegraphics[scale=0.15]{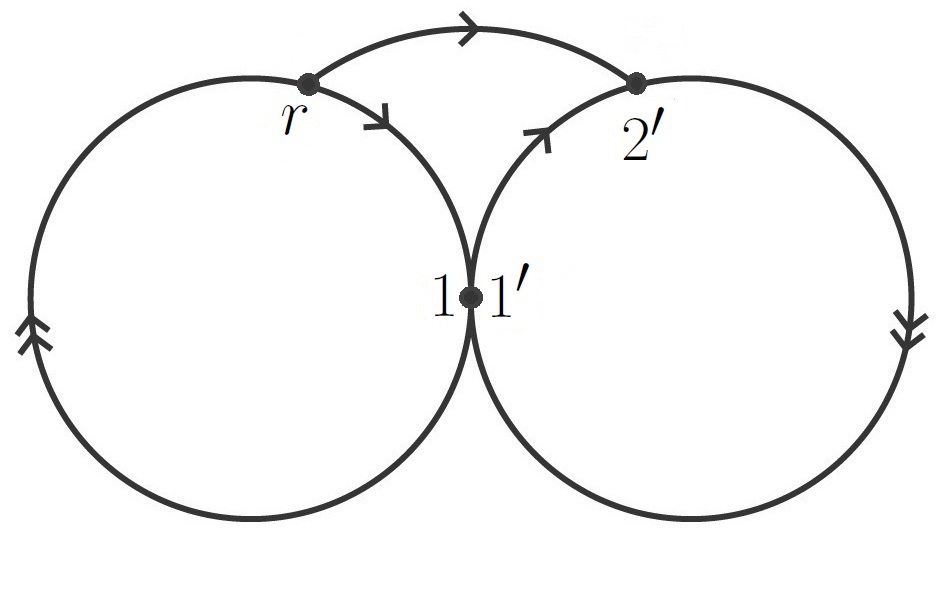}
\hspace{.5cm}
\includegraphics[scale=0.15]{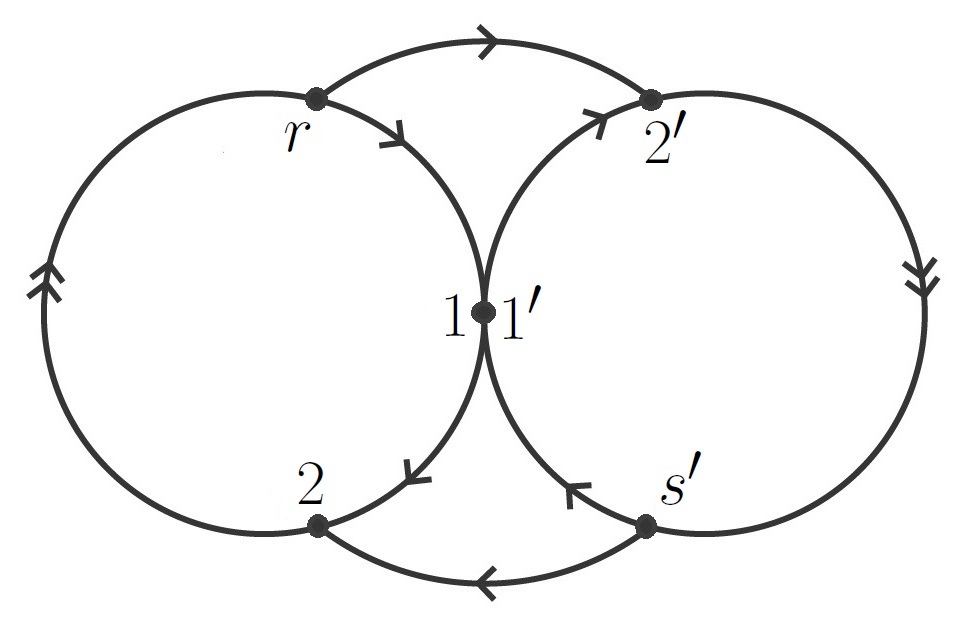}
\caption{Type 1 digraph (left) and Type 2 digraph (right)}.
\label{fig_tipo1and2}
\end{figure}


\textbf{Type 2 digraph}

Let $D_2$ be the digraph $D_2(r,s)=\infty(r,s)\cup\{e,\tilde{e}\}$, where $e=(r,2')$ and $\tilde{e}=(s',2)$, as shown in Figure \ref{fig_tipo1and2} (right). It is still true that the only strongly connected induced subdigraphs are cycles $\vec{C_r}$, $\vec{C_s}$ and $\vec{C}_{r+s-2}$ {(obtained by removing vertex $1=1'$)}, besides the digraph $D_2$ itself and isolated vertices. Then, we have
\[\Pi(D_2)=\{0,1,\rho(D_2)\}.\]

{Observe that these two types of digraphs, not only have $\infty$-subdigraphs, but also $\theta$-subdigraphs. For instance, for the Type 1 digraph, by removing the $(r,1)$ arc from $D_1$ we obtain a $\theta$-digraph, specifically a $\theta(0,r-1,s-2)$.}

\begin{obs}\label{rem:dcs}
The three examples presented above have three complementarity eigenvalues, namely, zero, one, and the spectral radius of the digraph itself. A direct application of Perron-Frobenius Theorem, gives us an important observation: if we consider the three digraphs mentioned (all with $n$ vertices), then, by Lemma \ref{lem:sub}, $\rho(\infty)<\rho(D_1)<\rho(D_2)$. In particular, these digraphs are not isomorphic to each other.
\end{obs}

\subsection{Digraphs with a $\theta$-subdigraph, and without an $\infty$-subdigraph.}
\label{sec:digrafos_con_teta}

We now describe three types of digraphs in $\mathcal{SCD}_3$
without $\infty$-subdigraphs. In the next section we prove that, up to isomorphism, they are the only ones with that property.\\


\noindent \textbf{Type 3 digraph}

Let $D_3$ be the digraph $D_3=\vec{C_n}\cup\{e,\tilde{e}\}$ where $e=(1,i)$ and $\tilde{e}=(i-1,j)$ with $2<i<j\leq n$, as shown in the left Figure \ref{fig:tipo345}.

The only strongly connected induced subdigraphs {of $D_3$} are the cycles $\vec{C}_{n-(i-2)}$, $\vec{C}_{n-(j-i)}$ besides the digraph $D_3$ and isolated vertices. Then, we have
\[\Pi(D_3)=\{0,1,\rho(D_3)\}.\]

\noindent \textbf{Type 4 digraph}\\
Let {$k$} be a positive integer larger than $1$. Consider the digraph $D_4=\vec{C_n}\cup\{e_1,e_2, \hdots, e_k\}$ where $e_i=(x_i,y_i)$ with $1<y_i<x_i<y_{i+1}<x_{i+1}\leq n$ for all $i=1, \hdots, k-1$. {Observe that the conditon $y_i<x_i$ prevents the appearence of an $\infty$-subdigraph.} Figure \ref{fig:tipo345} (center) shows an example of such a digraph {with $k=3$ added arcs, generating the three cycles $\vec{C}_{r_1}, \vec{C}_{r_2}$ and $\vec{C}_{r_3}$.}

The only strongly connected induced subdigraphs of $D_4$ are the cycles $\vec{C}_{r_1},\hdots,\vec{C}_{r_k}$  (where $C_{r_i}$ is the digraph induced {by the vertices in-between} $y_i$ and $x_i$ for all $i=1, \hdots, k$) besides the digraph $D_4$ and isolated vertices. Then, we have
\[\Pi(D_4)=\{0,1,\rho(D_4)\}.\]

\begin{figure}[h!]
\centering
\includegraphics[scale=0.15]{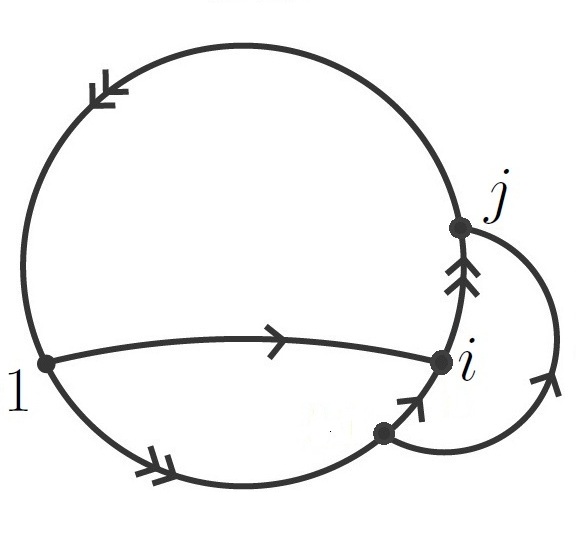}
\hspace{.5cm}
\includegraphics[scale=0.12]{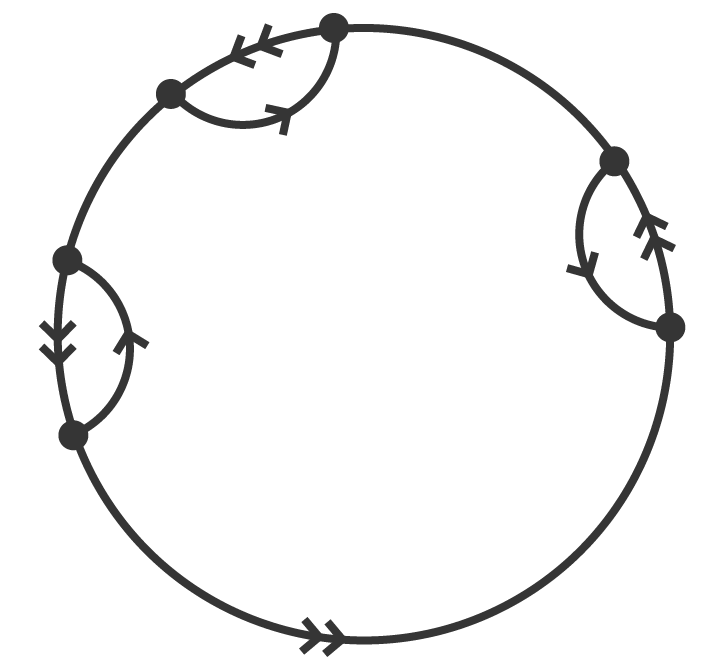}
\hspace{.5cm}
\includegraphics[scale=0.12]{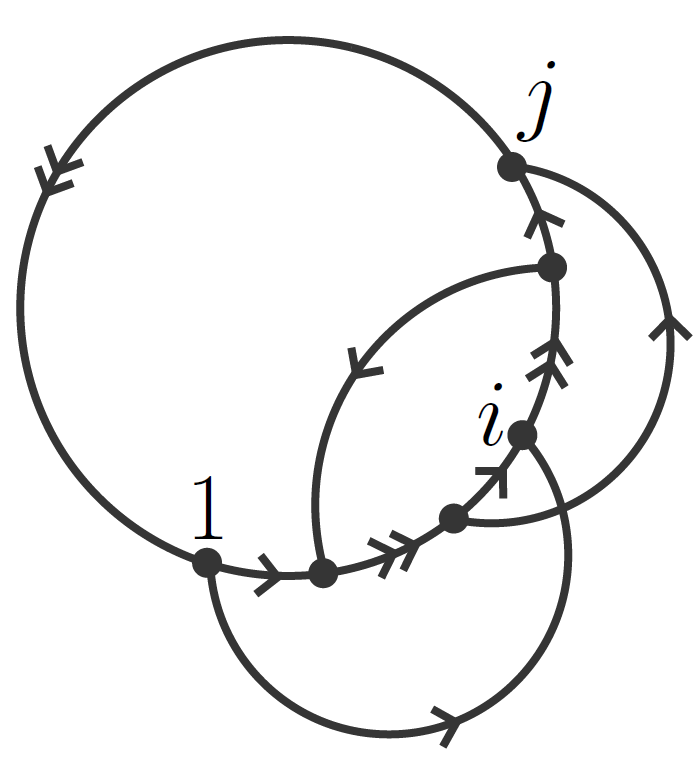}
\caption{Type 3 digraph (left), Type 4 digraph (center) and Type 5 digraph (right). {Observe that one $\theta$-subdigraph can be identified in these three examples by taking the larger (round) cycle in each case, and any other arc.}}
\label{fig:tipo345}
\end{figure}


\noindent \textbf{Type 5 digraph}

Consider the digraph $D_5=\vec{C_n}\cup\{e,e',e''\}$ where $e=(1,i)$, $e'=(i-1,j)$ and $e''=(j-1,2)$ with $3<i<i+1<j\leq n$. {In this case, condition $3<i$ as much as condition $i+1<j$ prevents the appearence of an $\infty$-subdigraph.} Figure \ref{fig:tipo345}(right) shows an example of a digraph in this family.

The only strongly connected induced subdigraphs {of $D_5$} are the cycles $\vec{C}_{n-(i-2)}$, $\vec{C}_{n-(j-i)}$ and $\vec{C}_{j-2}$ besides the digraph $D_5$ and isolated vertices. Then, we have
\[\Pi(D_5)=\{0,1,\rho(D_5)\}.\]

Note that the three types of digraphs defined above, as well as the $\theta$ digraph,{ do} not have an infinity digraph as a subdigraph.

{It is easy to see that these seven types of digraphs presented above are non-isomorphic to each other.}

\section{Characterization of $\mathcal{SCD}_3$}
\label{sec:char}

{Let $F_{\infty}$ be the set of digraphs which are either $\infty$-digraphs or Type $i$ digraphs with $i \in \{1, 2\}$, and let $F_{\theta}$ be the set of digraphs which are either $\theta$-digraphs or Type $j$ digraphs with $j \in \{3, 4, 5\}$. We refer to $F_{\infty}$ (resp.  $F_{\theta}$) as the $\infty$-Family (resp. the $\theta$-Family). In this section, we will prove that $\mathcal{SCD}_3 = F_{\infty} \cup F_{\theta}.$}


\subsection{$\infty$-Family characterization}
\label{subsec:infity}
We first show that the only digraphs in $\mathcal{SCD}_3$ having an $\infty$-subdigraph belong to $F_{\infty}$.

\begin{theorem}
 \label{subdigrafoocho}
Let $D$ be a digraph in  $\mathcal{SCD}_3$. If $D$ contains an $\infty$-subdigraph, then $D$ belongs to the $\infty$-Family. Precisely, $D$ is either an $\infty$-digraph, a Type 1 digraph or a Type 2 digraph.
\end{theorem}

\begin{proof}
Let us first suppose that there is a vertex $v$ in $D$ outside $\infty(r,s)$ (i.e. there are more vertices than those of the coalescence of the cycles). {Let $D'$ be the induced digraph obtained from $D$ by removing $v$. We notice that $D'$ contains an induced strongly connected proper subdigraph different from a cycle, which is $\infty(r,s)$ itself. This contradicts Theorem~\ref{caracterizacion}.}

We have then that the vertices in $D$ are exactly the ones of $\infty(r,s)$.
Let us analyze which arcs in $D$ are not in $\infty(r,s)$.

We recall that $\infty(r,s)=\vec{C_r}\cdot \vec{C_s}$. We first suppose that there is an arc between vertices of $\vec{C_r}$. Considering the digraph ${D'}$ generated by the vertices of $\vec{C_r}$, {then $D'$ is an induced strongly connected proper subdigraph different from a cycle or an isolated vertex, which contradicts Theorem \ref{caracterizacion}.}
{Figure~\ref{arcs_between_Cr} illustrates all possible arcs within $\vec{C}_r$ and an obtained subdigraph $D'$ in red that is not a cycle.}
\begin{figure}[h!]
\centering
\includegraphics[scale=0.12]{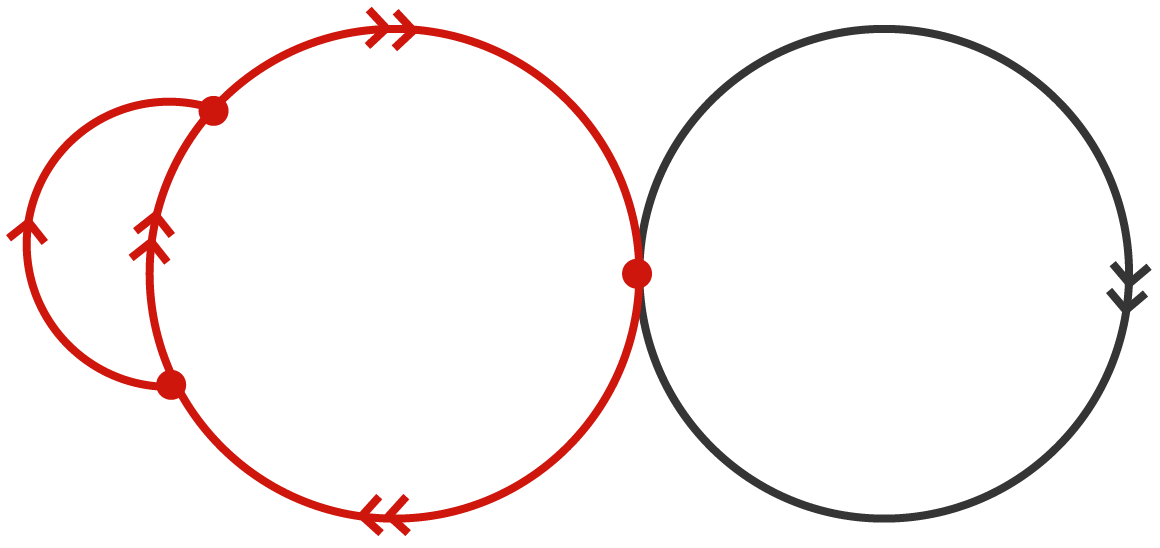}
\hspace{.5cm}
\includegraphics[scale=0.12]{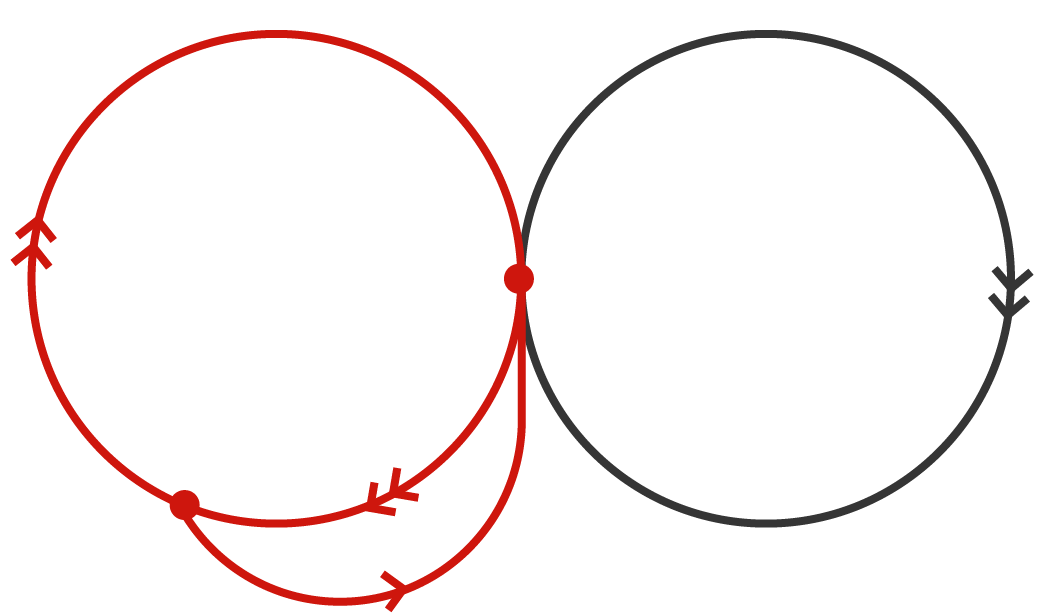}

\vspace{.5cm}
\includegraphics[scale=0.12]{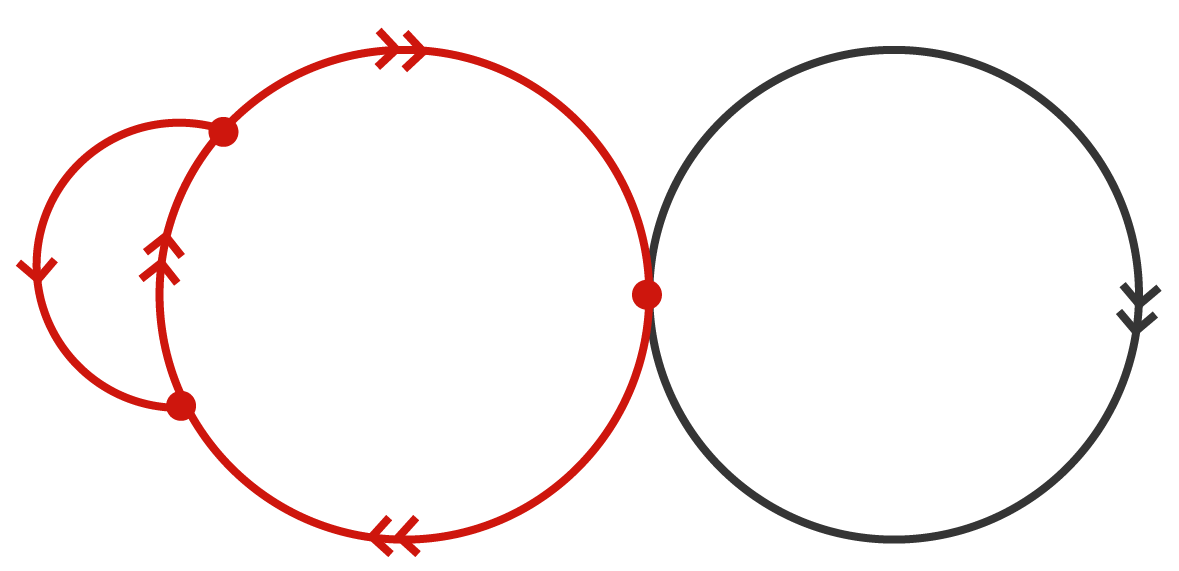}
\hspace{.5cm}
\includegraphics[scale=0.12]{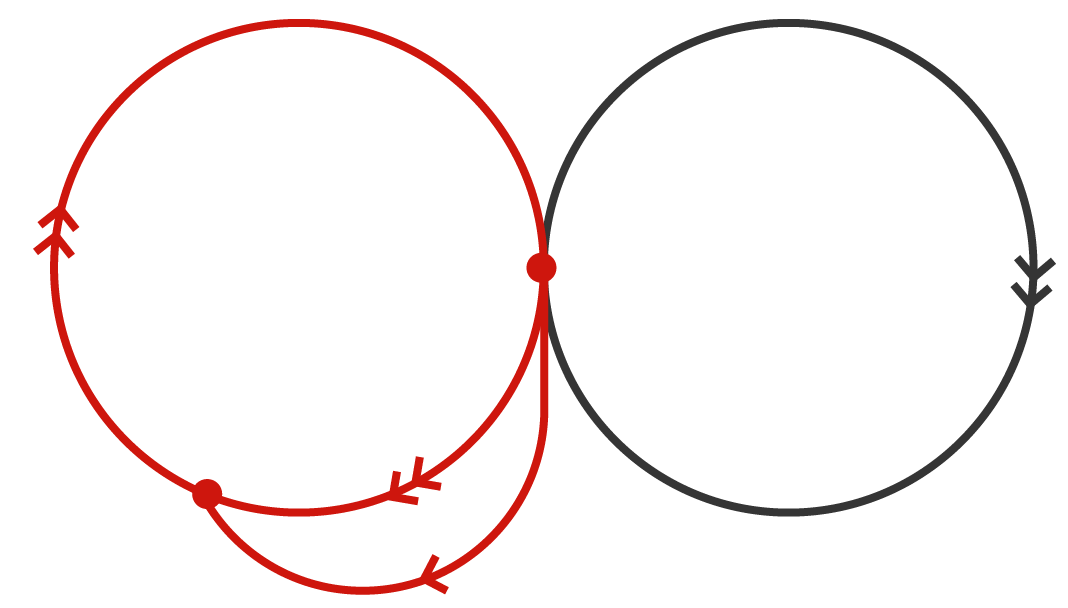}
\caption{A schematic representation of all possible cases of added arcs $(x,y)$ with $x,y \in V(\vec{C}_r)$.}
\label{arcs_between_Cr}
\end{figure}

Of course with the same argument we can rule out arcs between vertices of $\vec{C_s}$.
Therefore, we can only add arcs with one vertex in each cycle.

Let us now suppose that there is an arc $(x,y)$ with $x \in V(\vec{C_r})$ and $y \in V(\vec{C_s})$.

We will first see that, in order to maintain the cardinality of the complementarity spectrum, the arc must start from the vertex $r$ of $\vec{C_r}$.

Indeed, if $x\neq r$, we can consider the digraph $D'$ generated by the vertices of $D$ except $r$ (see the left of Figure~\ref{digraph_without_r}). Since $D'$ is an induced strongly connected proper subdigraph of $D$, {which contradicts Theorem \ref{caracterizacion}.}
Therefore the only way to keep three complementarity eigenvalues is to have $x=r$.
\begin{figure}[h!]
\centering
\includegraphics[scale=0.1]{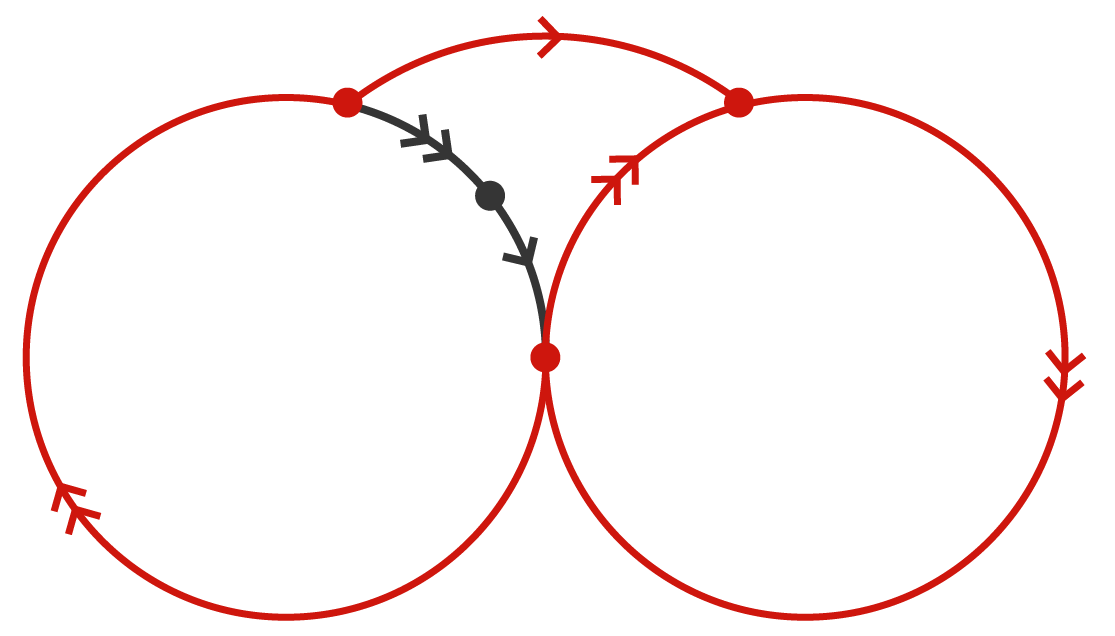}
\hspace{.5cm}
\includegraphics[scale=0.1]{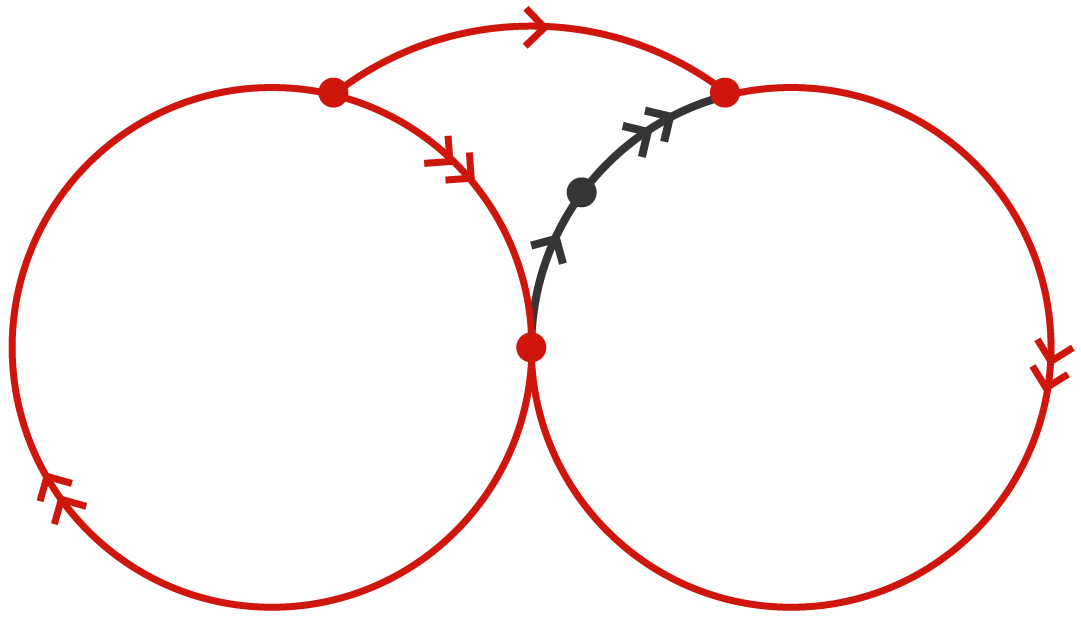}
\caption{Added arc $(x,y)$ with $x \in\vec{C_r}$, $y\in\vec{C_s}$ and $x\neq r$ (left). $(x,y)$ with $x \in\vec{C_r}$, $y\in\vec{C_s}$ and $y \neq 2'$ (right). }
\label{digraph_without_r}
\end{figure}

With the same arguments we can see that the arc has to end in the second vertex of $\vec{C_s}$, this is, $y=2'$ (see the right of Figure \ref{digraph_without_r}).

Then, the only arc that we can add from $\vec{C_r}$ to $\vec{C_s}$ maintaining three complementarity eigenvalues is $(x,y)=(r,2')$.

Analogously, we have that the only possible arc $(x,y)$ with $x \in V(\vec{C_s})$ and $y \in V(\vec{C_r})$ is $(s',2)$.

We have then three possibilities:
\begin{itemize}
\item we add no arc to the coalescence of the two cycles,
\item we add either $(r,2')$ or $(s',2)$, so we end up with a digraph of Type 1,
\item we add both arcs $(r,2')$ and $(s',2)$, so we end up with a digraph of Type 2,
\end{itemize}
which concludes the proof.
\end{proof}

\begin{cor}
Let $D$  be a digraph in  $\mathcal{SCD}_3$ with $n$ vertices. If $D$ contains an  $\infty (r,s)$-subdigraph, then $n=r+s-1$.
\end{cor}

\subsection{$\theta$-Family characterization}
\label{subsec:theta}

As we did in Subsection \ref{subsec:infity} with the $\infty$-Family, we will now show that the digraphs in $\theta$-Family presented above are all the possible digraphs in $\mathcal{SCD}_3$ {containing a $\theta$-subdigraph and not an $\infty$-subdigraph.}

\begin{teo}
\label{principal}Let $D$ be a digraph in  $\mathcal{SCD}_3$, with a $\theta$-subdigraph and without an $\infty$-subdigraph, then $D$ belongs to {the $\theta$-Family. Precisely $D$ is either a} $\theta$-digraph, a Type $3$ digraph, a Type $4$ digraph, or a Type $5$ digraph.
\end{teo}

\begin{proof}
Let us first suppose that there is a vertex $v$ in $D$ outside $\theta(a,b,c)$ (i.e. there are more vertices than those of the $\theta$-subdigraph). {Let $D'$ be generated from the vertices of $D$ by removing $v$}. We notice that $D'$ contains an induced strongly connected proper subdigraph different from a cycle or an isolated vertex, because it contains $\theta(a,b,c)$ as a proper subdigraph, and therefore it contradicts Theorem~\ref{caracterizacion}.

Then the vertices in $D$ are exactly the ones in $\theta(a,b,c)$.

We will analyze which arcs can be added to the digraph $\theta(a,b,c)$ using the following strategy: for any arc $(x,y)$ added to $\theta(a,b,c)$ we will try to find a non-trivial strongly connected induced subdigraph $D'$ different from a cycle and different from $D$. In the figures illustrating each case, the digraph $D'$ will be colored in red. If this is possible, then the arc $(x,y)$ can not be added because it would contradict Theorem~\ref{caracterizacion}. If not, the digraph obtained will be one of the four different types. Furthermore, we will need to guarantee that if we combine different permitted arcs we also obtain one of the four different types of digraphs.\\

\begin{figure}[h!]
\centering
\includegraphics[scale=0.15]{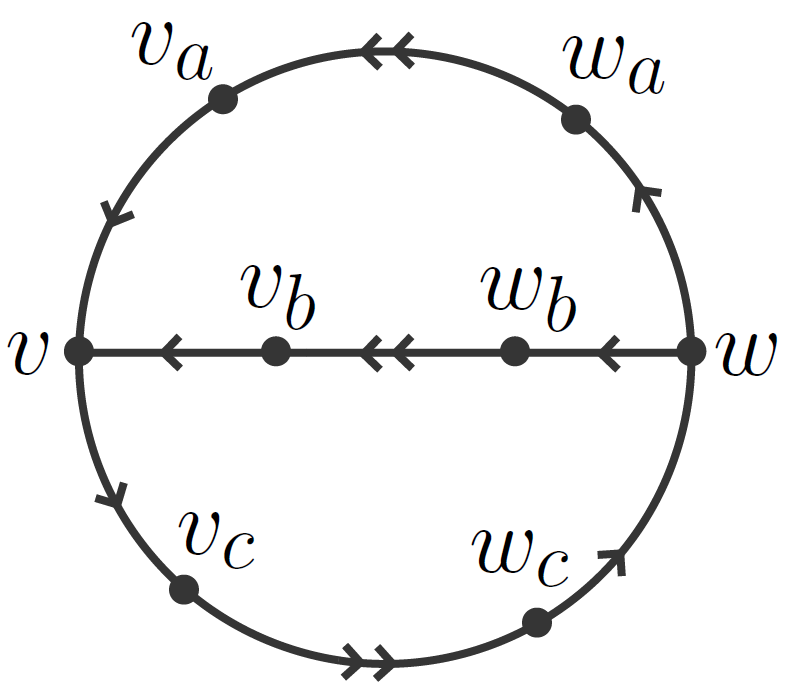}
\caption{$\theta$-digraph and some distinguished vertices in it.}
\label{grafos_40}
\end{figure}

For simplicity, since these vertices are used more than once in what follows, let us denote by $v_a=pred_{\vec{P}_a}(v)$, $v_b=pred_{\vec{P}_b}(v)$, and $v_c=suc_{\vec{P}_c}(v)$ the predecessor vertices of $v$ in the paths $\vec{P}_a$, $\vec{P}_b$ and $\vec{P}_c$, respectively. Analogously, we denote, respectively, by  $w_a=suc_{\vec{P}_a}(w)$, $w_b=suc_{\vec{P}_b}(w)$, $w_c=pred_{\vec{P}_c}(w)$ the successor vertices of $w$. See the representation in Figure~\ref{grafos_40}.

For the analysis it will be useful to analyze cases $a \geq 1$ and $a=0$ separately. Remember that, from the definition of $\theta$-digraphs in Section \ref{sec:basic_families} we are assuming $b \geq 1$.\\

\noindent\underline{\textbf{Case 1}} $a \geq 1$\\
We will further separate the study in five subcases.
\begin{enumerate}
\item $x,y \in V(\vec{P}_{a+2})\cup V(\vec{P}_{c+2})$,
\item $x,y \in V(\vec{P}_{b+2})$,
\item $x \in V(\vec{P}_{b+2})$ and $y \in V(\vec{P}_{c+2})$ or viceversa,
\item $x \in V(\vec{P}_{a+2})$ and $y \in V(\vec{P}_{b+2})$,
\item $x \in V(\vec{P}_{b+2})$ and $y \in V(\vec{P}_{a+2})$.
\end{enumerate}

\noindent \textbf {Subcase (1).} We will prove that we cannot add arcs $(x,y)$ with $x,y \in V(\vec{P}_{a+2})\cup V(\vec{P}_{c+2})$.\\
Indeed, let us suppose that we added such an arc. Then, by considering the digraph ${D'}$ generated by $V(\vec{P}_{a+2})\cup V(\vec{P}_{c+2})$ (see Figure \ref{grafos_09} (left and center), where $D'$ is in red) we obtain an induced strongly connected proper subdigraph  different from a cycle and from an isolated vertex which contradicts Theorem \ref{caracterizacion}.

\begin{figure}[h!]
\centering
\includegraphics[scale=0.1]{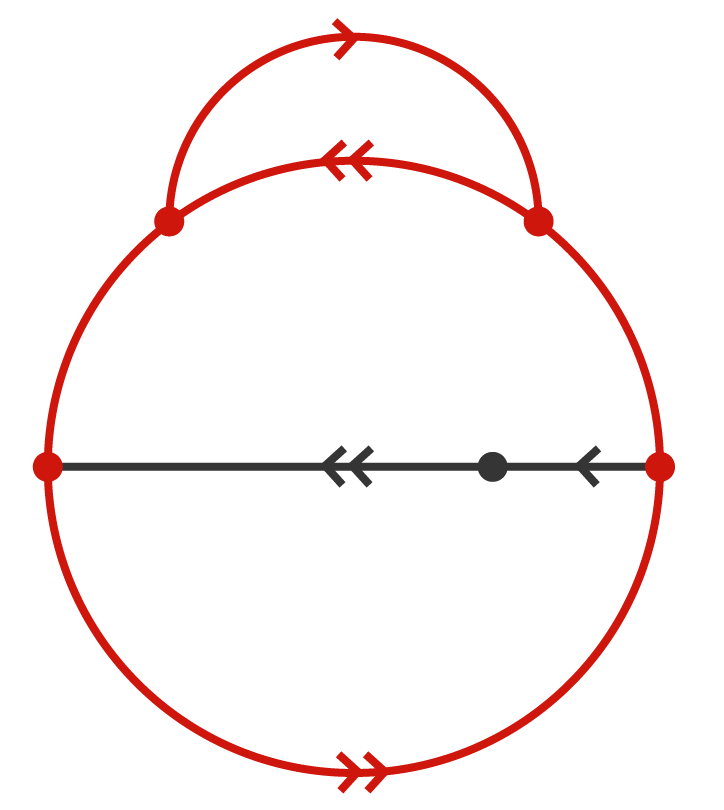}
\hspace{1cm}
\includegraphics[scale=0.1]{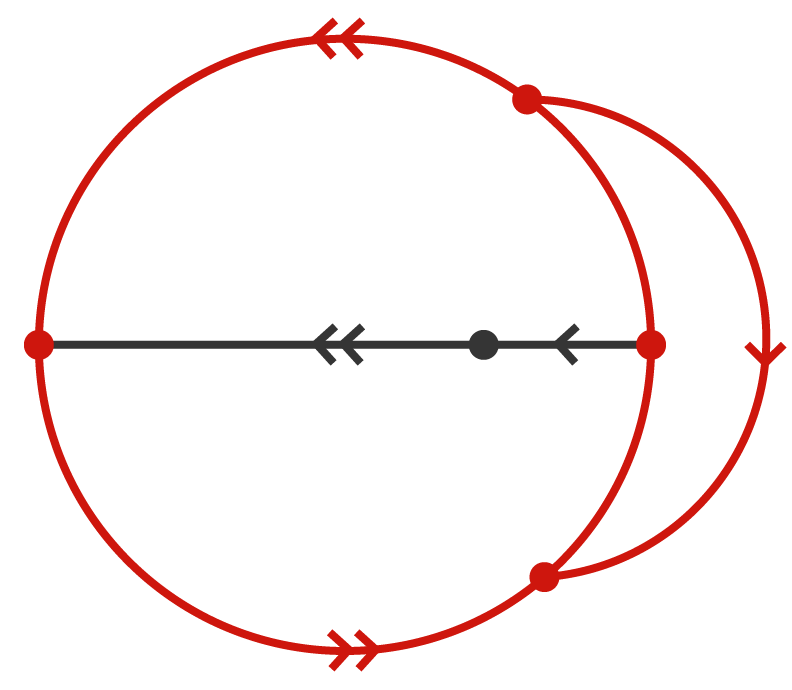}
\hspace{1cm}
\includegraphics[scale=0.1]{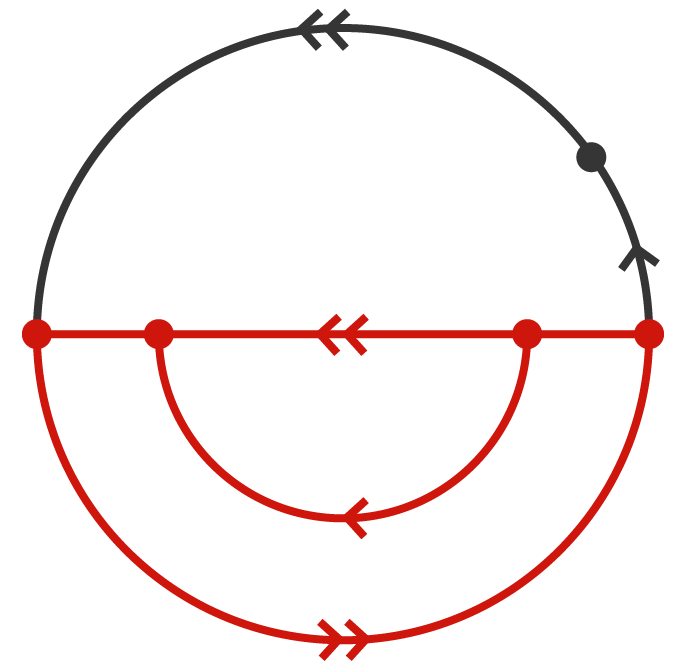}
\caption{Added arc $(x,y)$ with $x,y \in V(\vec{P}_{a+2})\cup V(\vec{P}_{c+2})$ (left and center). Added arc $(x,y)$ with $x,y \in V(\vec{P}_{b+2})$ (right). }
\label{grafos_09}
\end{figure}

Then there are no arcs between vertices of $\vec{P}_{a+2}$ or vertices of $\vec{P}_{c+2}$. {The same proof could be done with $a=0$.}

\noindent \textbf{Subcase (2).} We will prove that we cannot add arcs $(x,y)$ with $x,y \in V(\vec{P}_{b+2})$. Indeed, let us suppose that we added such an arc. Hence, considering the digraph ${D'}$ generated by $V(\vec{P}_{b+2}) \cup V(\vec{P}_{c+2})$ (see Figure \ref{grafos_09} (right)) {we obtain an induced strongly connected proper subdigraph different from a cycle and from an isolated vertex which contradicts Theorem \ref{caracterizacion}.}

Then there are no arcs between vertices of $\vec{P}_{b+2}$.
{Note that we used that $a\geq 1$.}\\

\noindent \textbf{Subcase (3)}
 We will prove that we can neither add arcs $(x,y)$ with $x \in V(\vec{P}_{b+2})$ and $y \in V(\vec{P}_{c+2})$ nor viceversa. Indeed, let us suppose that we added such an arc. Hence, considering the digraph ${D'}$ generated by $V(\vec{P}_{b+2}) \cup V(\vec{P}_{c+2})$ (see Figure \ref{grafos_10}) {we obtain an induced strongly connected proper subdigraph different from a cycle and from an isolated vertex which contradicts Theorem \ref{caracterizacion}.}

\begin{figure}[h!]
\centering
\includegraphics[scale=0.1]{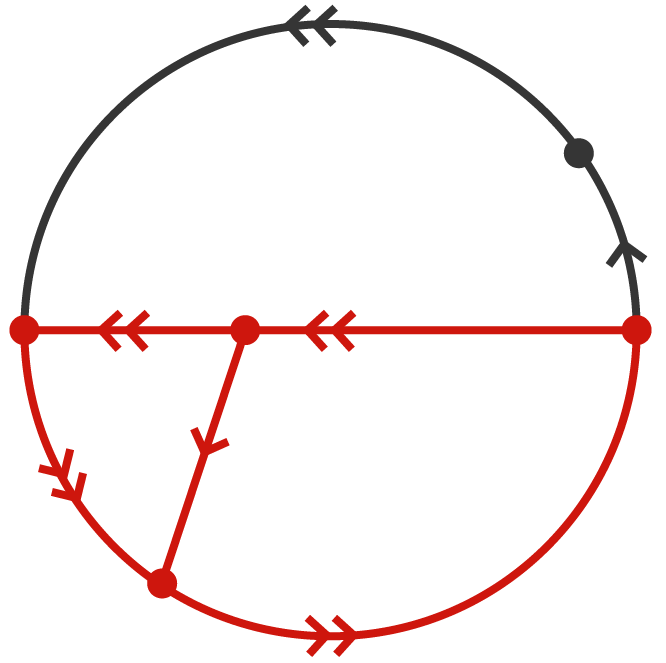}
\hspace{.6cm}
\includegraphics[scale=0.1]{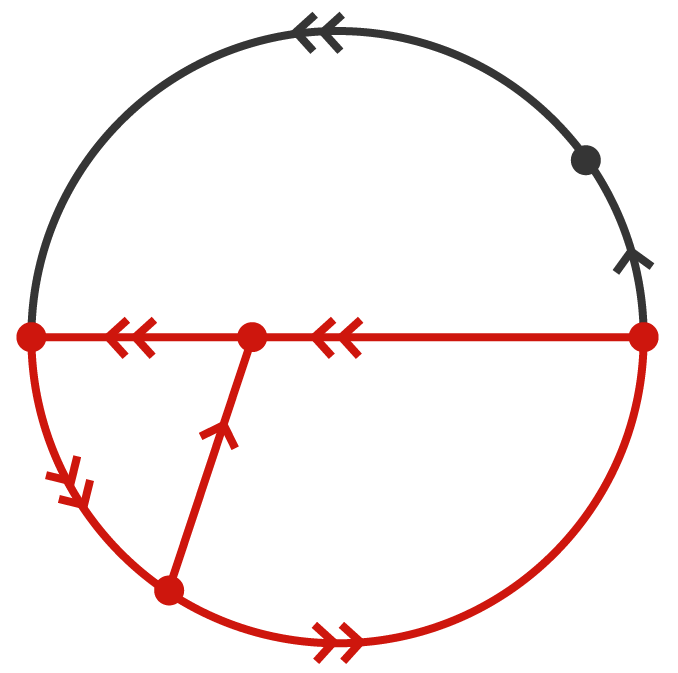}
\caption{Added arc $(x,y)$, with $x\in V(\vec{P}_{a+2})$ and $y\in V(\vec{P}_{c+2})$, and viceversa.}
\label{grafos_10}
\end{figure}
Then there are no arcs $(x,y)$ with $x \in V(\vec{P}_{b+2})$ and $y \in V(\vec{P}_{c+2})$ or viceversa that keep the number of complementarity eigenvalues.\\

\noindent \textbf{Subcase (4).} We will prove that the only arc $(x,y)$ with $x \in V(\vec{P}_{a+2})$ and $y \in V(\vec{P}_{b+2})$ that could be added is $({v}_a,{w}_b)$, {where the vertices $v,~w.~v_a,~w_b$ are located in $D$ as Figure~\ref{grafos_40} illustrates.}\\

We can assume $x\neq v$ and $y \neq w$, otherwise we obtain {a subcase} already analyzed.\\

Let us first see that the starting vertex $x$ of the arc cannot be other than $v_a$. Indeed, if $x \neq v_a$ considering the strongly connected digraph ${D'}$ generated by removing vertex $v_a$ (see the left of Figure \ref{grafos_12}), {we obtain an induced strongly connected proper subdigraph different from a cycle and from an isolated vertex which contradicts Theorem \ref{caracterizacion}.}, so that $x={v_a}$.\\
\begin{figure}[h!]
\centering
\includegraphics[scale=0.1]{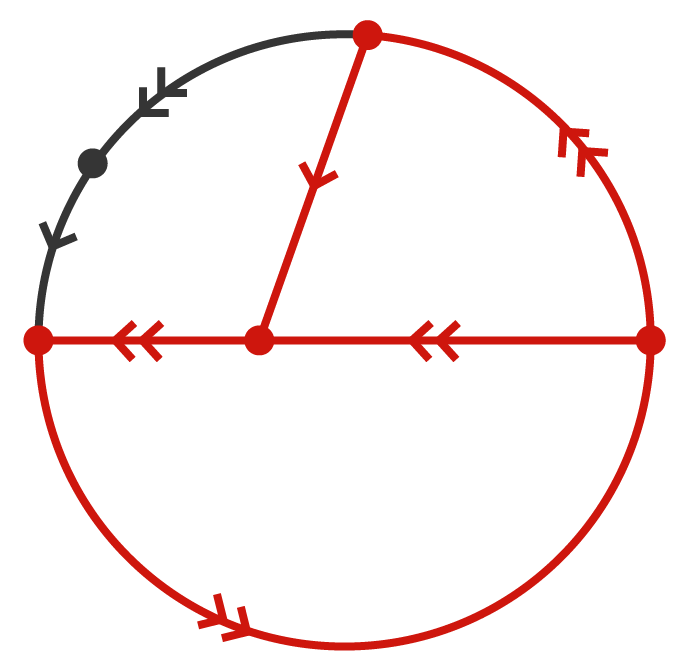}
\hspace{.6cm}
\includegraphics[scale=0.1]{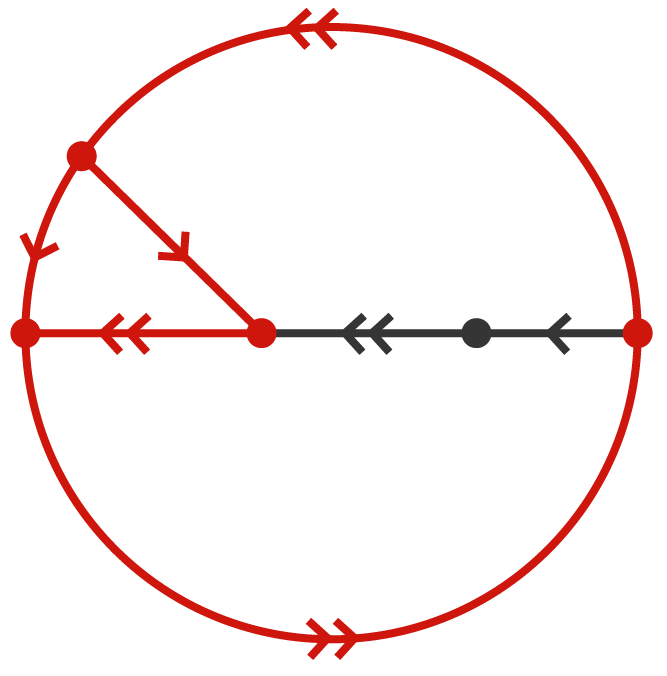}
\caption{Added arc $(x,y)$, with $x\in V(\vec{P}_{a+2})$, $x\neq v_a$ and $y\in V(\vec{P}_{b+2})$ (left). Added arc $(v_a,y)$, with $y\in V(\vec{P}_{b+2})$ and $y\neq w_b$ (right).}
\label{grafos_12}
\end{figure}

Analogously, the ending vertex $y$ of the added arc can only be $w_b$. Indeed, if $x={v}_a$ and $y \neq {w}_b$ considering the strongly connected digraph ${D'}$ generated by removing vertex ${w}_b$ (see the right of Figure~\ref{grafos_12}) {we obtain an induced strongly connected proper subdigraph different from a cycle and from an isolated vertex which contradicts Theorem \ref{caracterizacion}}, so that $y={w}_b$.

{We conclude that} the only arc we are able to add in this item is $({v}_a,{w}_b)$ and we obtain a Type $3$ digraph (see Figure \ref{grafos_11a}).
\begin{figure}[h!]
\centering
\includegraphics[scale=0.12]{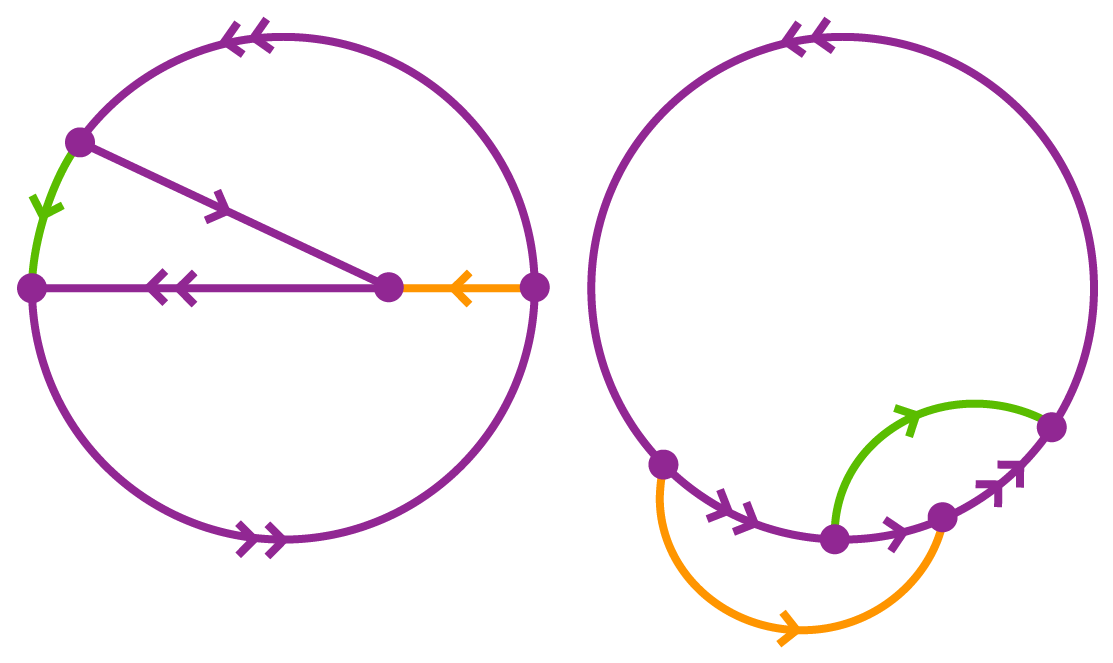}
\caption{Added arc $({v}_a,{w}_b)$ and its conversion into a Type $3$ digraph}
\label{grafos_11a}
\end{figure}

\newpage
\noindent \textbf{Subcase (5).} Analogously to the previous item we can prove that the only arc $(x,y)$ with $x \in V(\vec{P}_{b+2})$ and $y \in V(\vec{P}_{a+2})$ that could be added is $({v}_b, w_a)$ and we obtain a Type $3$ digraph (see Figure \ref{grafos_11b}).\\

\begin{figure}[h!]
\centering
\includegraphics[scale=0.13]{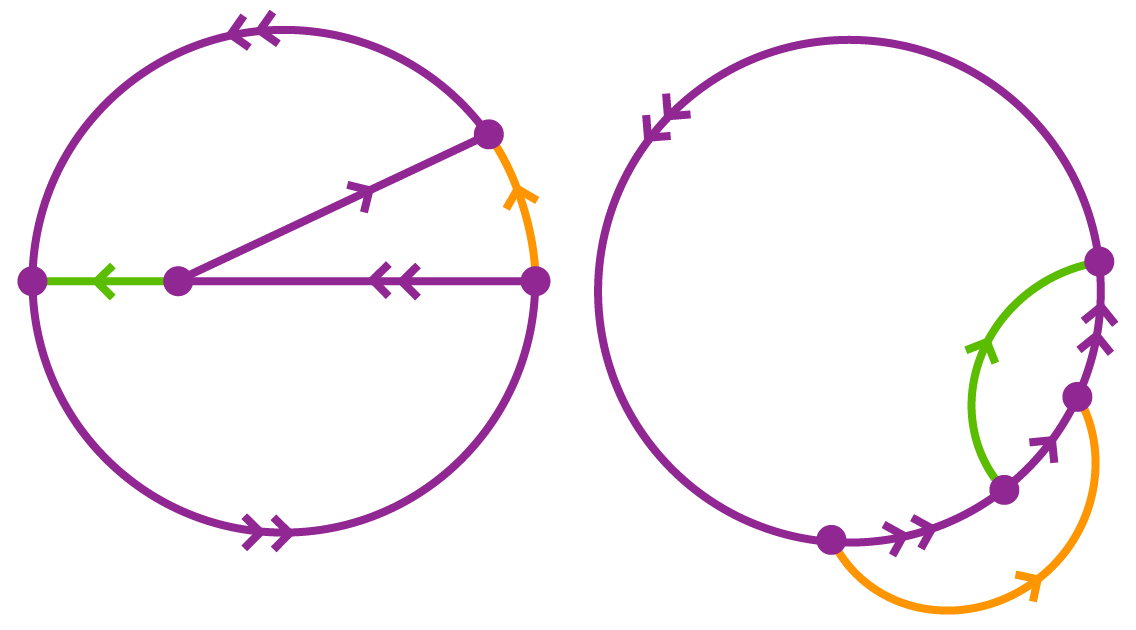}
\caption{Added arc $({v}_b,{w}_a)$ and its conversion into a Type $3$ digraph.}
\label{grafos_11b}
\end{figure}

We have proved so far that in case $D$ has a $\theta(a,b,c)$-subdigraph with $a\geq 1$,
 the only arcs that could be added are $({v}_a,{w}_b)$ and $({v}_b,{w}_a)$. If we add one of them we obtain a type $3$ digraph as seen before and if we add both we obtain a Type $5$ digraph as shown in Figure \ref{grafos_11}.\\

\begin{figure}[h!]

\centering
\includegraphics[scale=0.18]{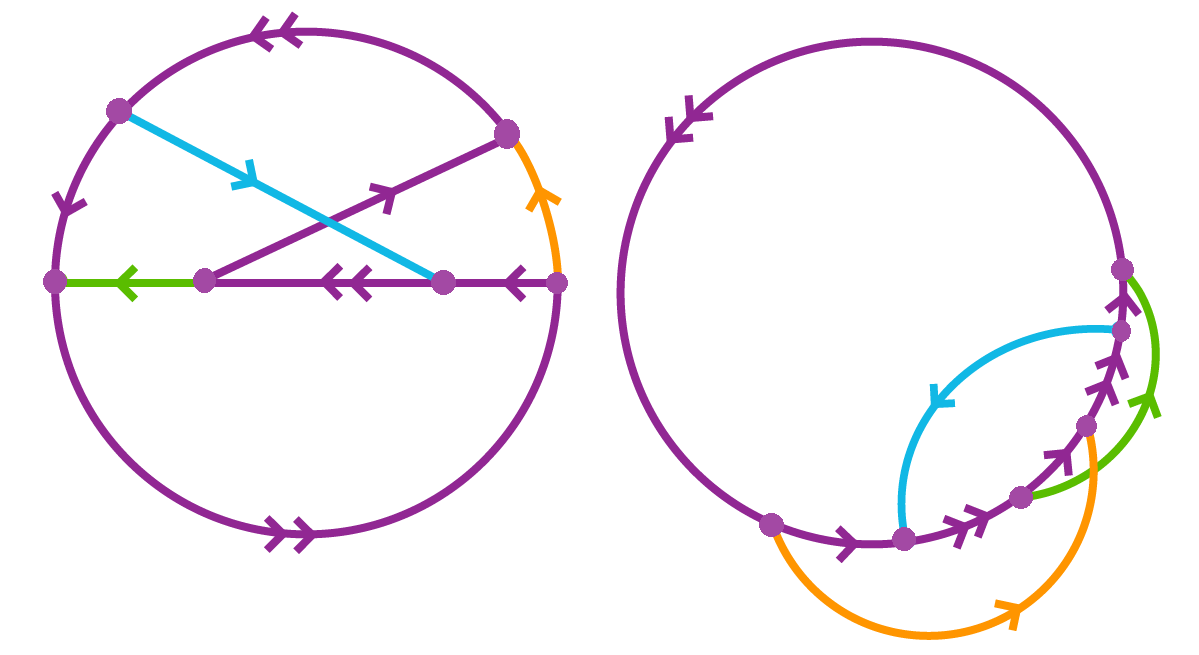}
\caption{Added arcs $({v}_a,{w}_b)$ and $({v}_b,{w}_a)$, and its conversion into a Type $5$ digraph.}
\label{grafos_11}
\end{figure}


\noindent\textbf{\underline{Case 2}~~$a=0$}\\
{We will analyze separetely the arcs $(x,y)$ with
\begin{enumerate}
\item $x,y \in V(\vec{P}_{a+2})\cup V(\vec{P}_{c+2})$,
\item $x,y \in V(\vec{P}_{b+2})$,
\item $x \in V(\vec{P}_{a+2})$ and $y \in V(\vec{P}_{b+2})$ or viceversa,
\item $x \in V(\vec{P}_{b+2})$ and $y \in V(\vec{P}_{c+2})$,
\item $x \in V(\vec{P}_{c+2})$ and $y \in V(\vec{P}_{b+2})$.\\
\end{enumerate}}

\noindent \textbf{Subcase (1).}
Due to the observation made in subcase  (1) of the previous case we know that there could not be arcs between vertices in $V(\vec{P}_{a+2}) \cup V(\vec{P}_{c+2})$ despite $a=0$.\\

\noindent \textbf{Subcase (2).}
Consider in $\vec{P}_{b+2}$ the natural order (in the figures, from right to left). We will prove that if an arc $(x,y)$ with $x,y \in V(\vec{P}_{b+2})$ can be added, its vertices should verify that $v<y<x<w$, {where $v$ and $w$ are located as in Figure~\ref{grafos_40}}. Furthermore, we will prove that if we add $k$ of these arcs $(x_i,y_i)$ with  $i=1, \hdots, k$ they should verify $w<y_i<x_i<y_{i+1}<x_{i+1}<v$ for all $i=1, \hdots, k$.\\

If $x < y$ there exists $z$ such that $x<z<y$, otherwise we would have added an existing arc. Considering the strongly connected digraph ${D'}$ generated by removing vertex $z$  we obtain an induced strongly connected proper subdigraph different from a cycle and from an isolated vertex which contradicts Theorem~\ref{caracterizacion}. This is shown in red in Figure~\ref{grafos_16} (left).
    \begin{figure}[h!]
    \centering
     \includegraphics[scale=0.1]{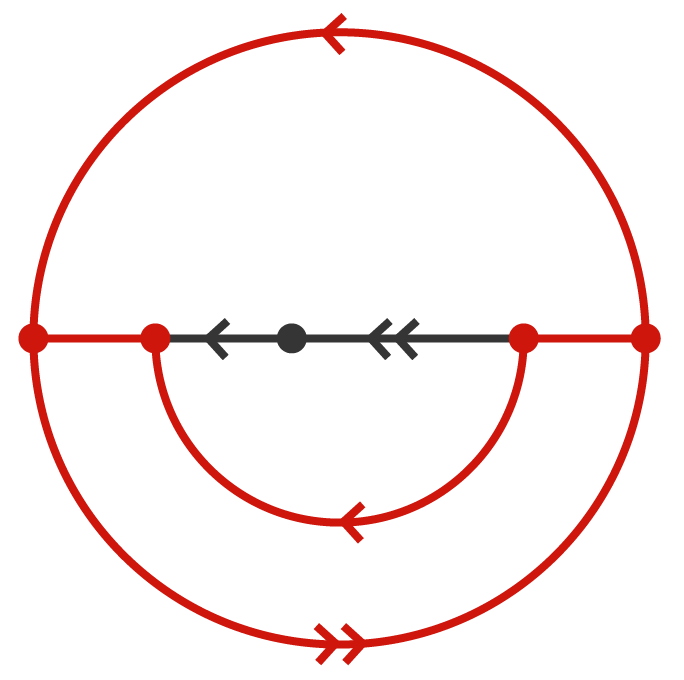}
     \includegraphics[scale=0.12]{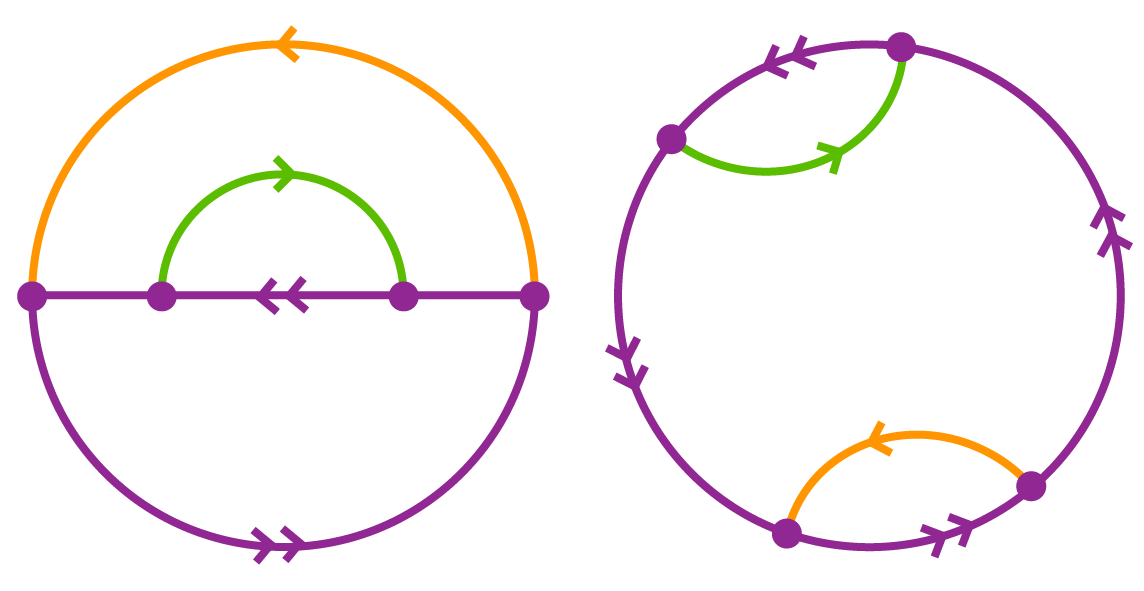}
    \caption{Added arc $(x,y)$, with $x\in V(\vec{P}_{b+2})$, $x<y$ (left). Added arc
    $(x,y)$, with $x\in V(\vec{P}_{b+2})$, $w< y < x < v$ (center) and its conversion to a Type $4$ digraph (right).}
    \label{grafos_16}
    \end{figure}

If $y < x$ we have that $w< y < x < v$ because if either $y=w$ or $x=v$, we would have an $\infty$-subdigraph, which was excluded in the hypothesis. Then, we obtain a Type 4 subdigraph as shown in the {Figure~\ref{grafos_16} (center). The right of Figure~\ref{grafos_16} shows how the obtained digraph may be seen as a Type 4.} Notice that the case $y=w$ and $x=v$ was analyzed in subcase (1).

Let us analyze what happens if we add another arc $(x',y')$ of this form. The vertices of these two arcs should verify one of the following three options
\begin{itemize}
    \item $w<y \leq y'< x' \leq x <v$,
    \item $w<y<y'<x < x'<v$,
    \item $w<y<x<y'< x'<v$.
\end{itemize}

If $w<y \leq y'< x' \leq x <v$,  considering the digraph $D'$ generated by vertices between $y$ and $x$ (including both), as shown in left of Figure \ref{grafos_37}, we obtain a digraph different from a cycle and from $D$ itself, what contradicts {Theorem~\ref{caracterizacion}}.

\begin{figure}[h!]
\centering
\includegraphics[scale=0.1]{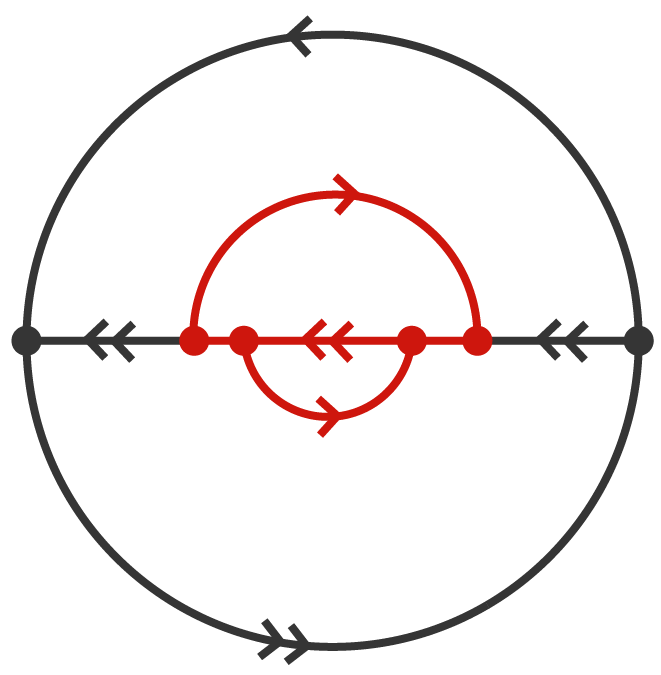}
\hspace{.5cm}
\includegraphics[scale=0.1]{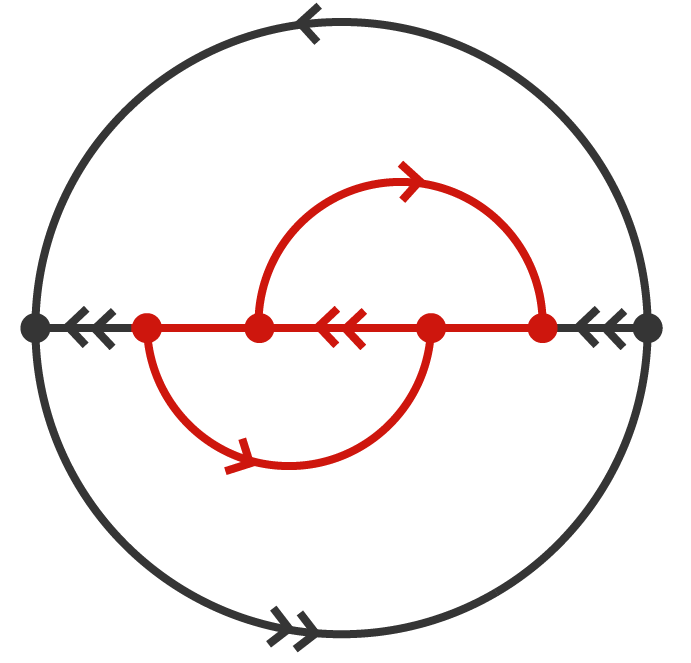}
\caption{Added arcs $(x,y)$ and $(x',y')$, $w<y \leq y'< x' \leq x <v$ (left). Added arcs $(x,y)$ and $(x',y')$ with $w<y \leq y'< x' \leq x <v$ (right).}
\label{grafos_37}
\end{figure}

If $w<y<y'\leq x < x'<v$  considering the digraph $D'$ generated by vertices between $y$ and $x'$ (including both), as shown in right of Figure~\ref{grafos_37}, {we obtain an induced strongly connected proper subdigraph different from a cycle and from an isolated vertex which contradicts Theorem~\ref{caracterizacion}.}

Then we have that these arcs necessarily verify $w<y<x<y'< x'<v$. If we add $k$ of these arcs, $(x_i,y_i)$ with $i=1, \hdots, k$, then $w<y_i<x_i<y_{i+1}<x_{i+1}<v$ for all $i=1, \hdots, k$ (reordering if necessary) which is a Type 4 digraph as shown in Figure~\ref{grafos_36}.\\

\begin{figure}[h!]
\centering
\includegraphics[scale=0.15]{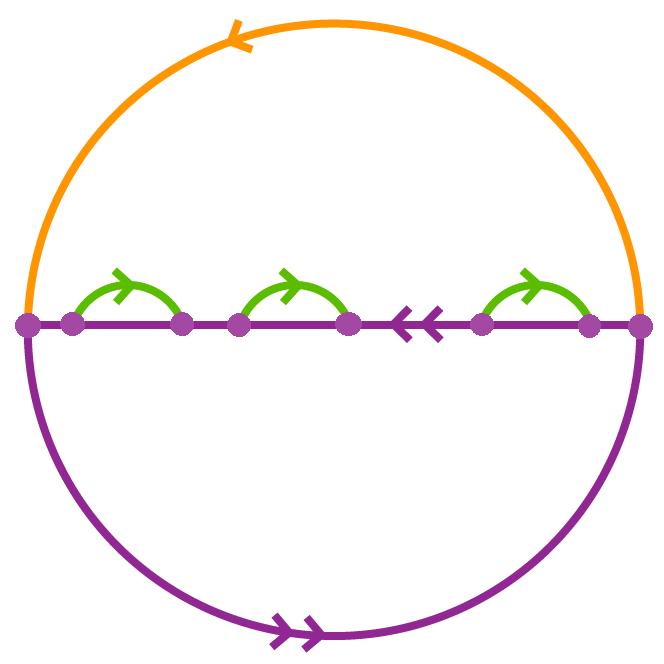}
\hspace{.5cm}
\includegraphics[scale=0.15]{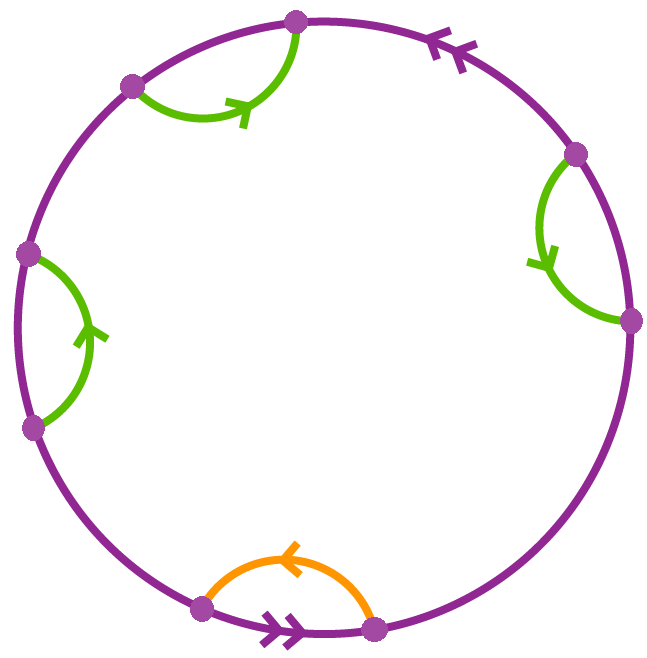}
\caption{Added arcs $(x_i,y_i)$ with $w<y_i<x_i<y_{i+1}<x_{i+1}<v$, $i=1, \hdots, k$ and its conversion into a Type $4$ digraph.}
\label{grafos_36}
\end{figure}

\noindent \textbf{Subcase (3).}
This item is reduced to subcases (1) or (2) because, since $a=0$, then $\vec{P}_{a+2}$ contains only two vertices, which can  be seen as vertices in $\vec{P}_{b+2}$.

\noindent \textbf{Subcase (4).}
 We will prove that if an arc $(x,y)$ with $x \in V(\vec{P}_{b+2})$ and $y \in V(\vec{P}_{c+2})$ can be added, then $x=v_b$ {(see Figure~\ref{grafos_40} for location of the vertex $v$)} and that not any other arc of these type can be added.

Indeed, let us suppose that we added such an arc $(x,y)$ with $x\neq {v_b}$. Hence, considering the strongly connected digraph $D'$ generated by vertices in $D$ different from $v_b$, as shown in Figure \ref{grafos_56} (left), we obtain a digraph different from a cycle and from $D$ itself, which contradicts {Theorem~\ref{caracterizacion}.}

\begin{figure}[h!]
\centering
\includegraphics[scale=0.1]{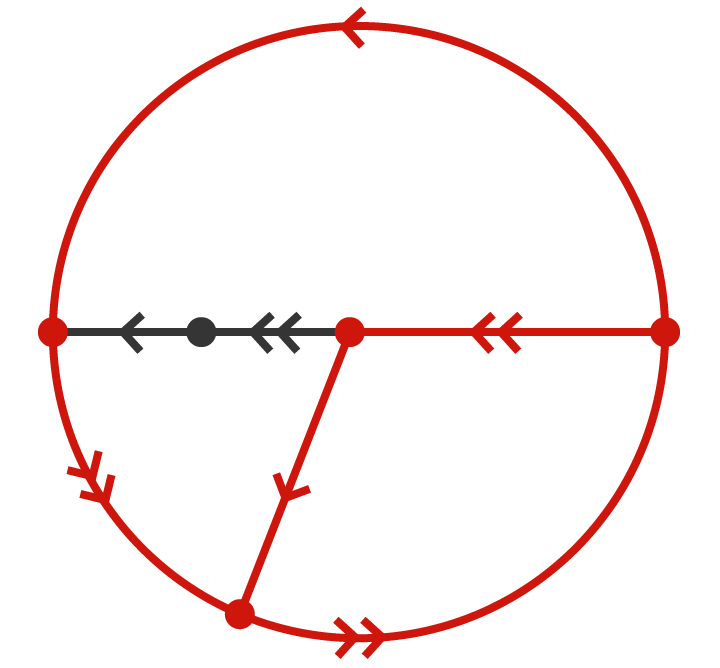}
\includegraphics[scale=0.14]{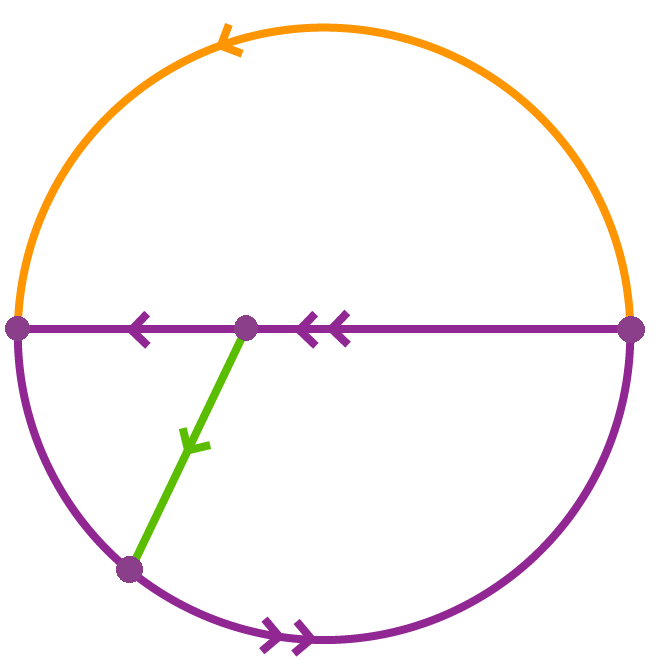}
\includegraphics[scale=0.14]{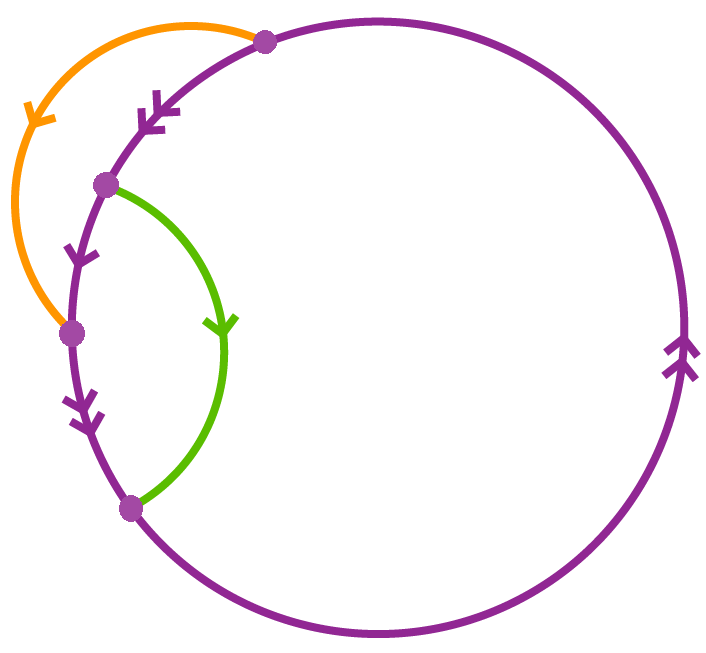}

\caption{Added arc $(x,y)$, with $x\in V(\vec{P}_{b+2})$, $x\neq v_b$ and $y\in V(\vec{P}_{c+2})$ (left). Added arc $(v_b,y)$ with $y \in \vec{P}_{c+2}$ (center) and its conversion into a Type $3$ digraph (right).}
\label{grafos_56}
\end{figure}

Moreover, if $x={v}_b$ then $y \neq v$ because $D$ has no multiple arcs, and $y \neq w $ because $D$ has not an $\infty$-subdigraph and we obtain a Type 3 digraph as shown in the Figure \ref{grafos_56} (center).

Notice that just one of these arcs could be added, otherwise, considering the strongly connected digraph $D'$ generated by vertices in $D$ different from $v$ (see Figure \ref{grafos_56b} (left)) {we obtain an induced strongly connected proper subdigraph different from a cycle and from an isolated vertex which contradicts Theorem~\ref{caracterizacion}.} \\

\begin{figure}[h!]
\centering
\includegraphics[scale=0.12]{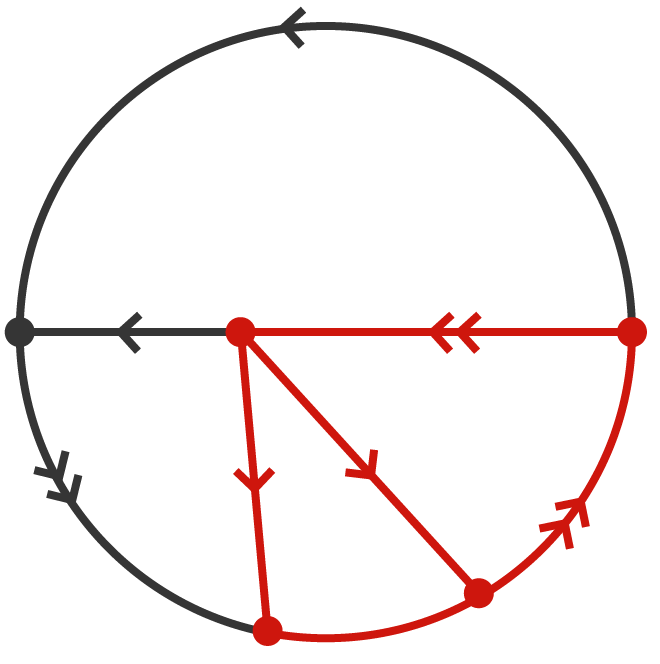}
\hspace{0.5cm}
\includegraphics[scale=0.12]{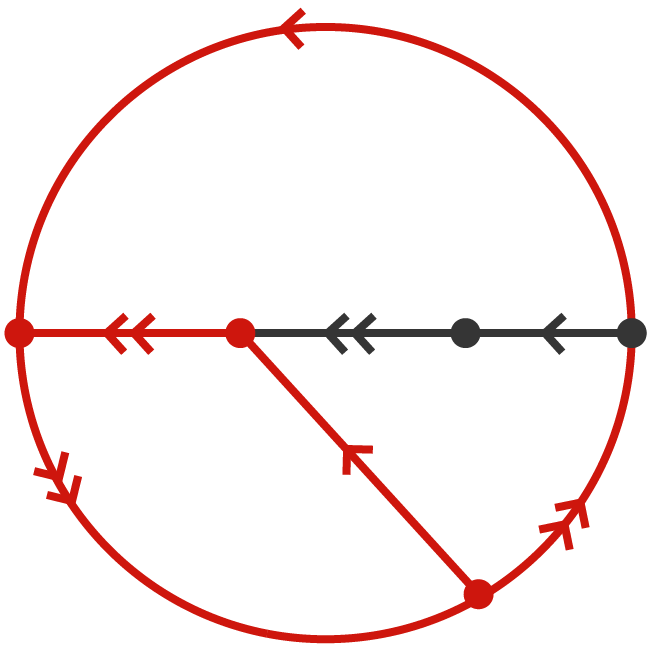}
\caption{Added arcs $(v_b,y)$ and $(v_b,y')$  with $x \in V(\vec{P}_{c+2})$ (left). Added arc $(x,y)$, with $x \in V(\vec{P}_{c+2})$, $x \in V(\vec{P}_{b+2})$ $y\neq w_b$ (right).  }
\label{grafos_56b}
\end{figure}

\noindent \textbf{Subcase (5).}
We will prove that if an arc $(x,y)$ with $x \in V(\vec{P}_{c+2})$ and $y \in V(\vec{P}_{b+2})$ can be added, then $y=w_b$ and that no other arc of this type can be added.

Indeed, let us suppose that we added such an arc $(x,y)$ with $y\neq {w_b}$. Hence, considering the strongly connected digraph $D'$ generated by removing ${w_b}$ (see Figure \ref{grafos_56b}) {we obtain an induced strongly connected proper subdigraph different from a cycle and from an isolated vertex which contradicts Theorem~\ref{caracterizacion}.} Then $y = w_b$.

If $y={w}_b$ then $x \neq w$ because $D$  has no multiple arcs, and $x \neq v $ because $D$ does not have an $\infty$-subdigraph, then we obtain a Type $3$ digraph as shown in Figure \ref{grafos_50a} (left and center).

\begin{figure}[h!]
\centering
\includegraphics[scale=0.15]{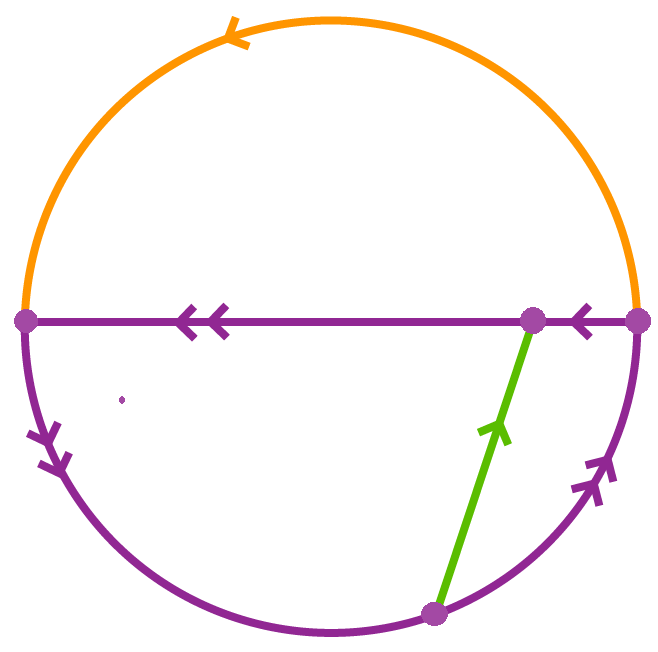}
\includegraphics[scale=0.15]{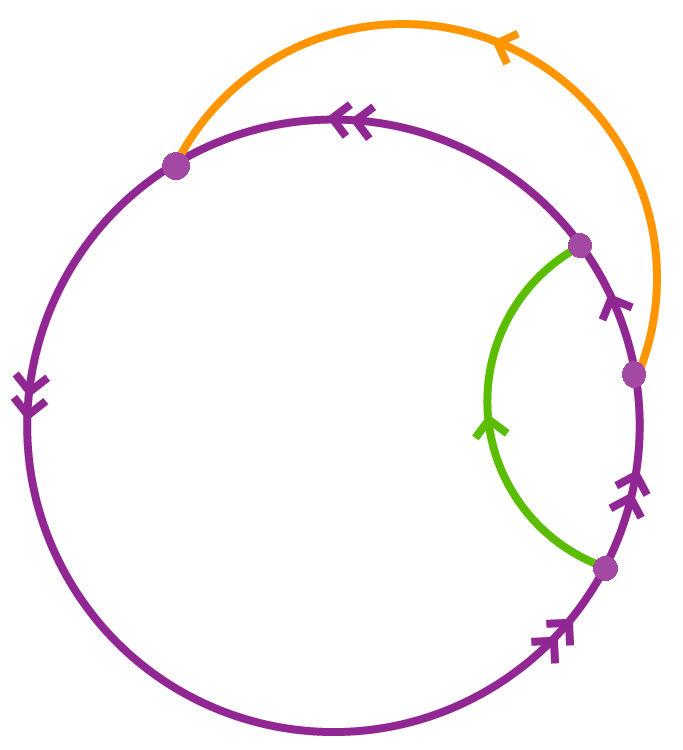}
\includegraphics[scale=0.12]{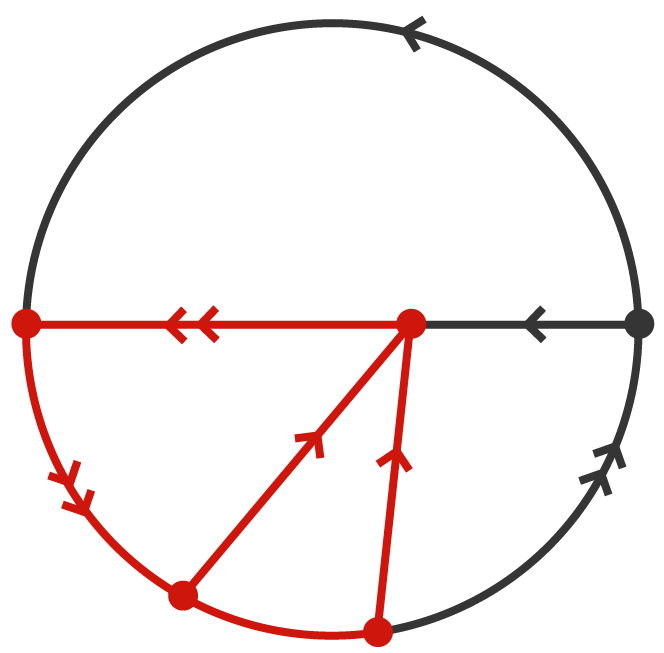}
\caption{Added arc ($x,w_b)$ (left) and its conversion to a Type $3$ (center). Added arcs $(x,w_b)$ and $(x',w_b)$ (right).}
\label{grafos_50a}
\end{figure}

Notice that only one of these arcs could be added, otherwise, considering the strongly connected digraph $D'$ generated by removing $v$ (see Figure \ref{grafos_50a} (right)), we obtain an induced strongly connected proper subdigraph different from a cycle and from an isolated vertex which contradicts Theorem~\ref{caracterizacion}.

We have proved so far that in case $D$ has a $\theta(0,b,c)$-subdigraph, the only arcs that can be added, maintaining three complementarity eigenvalues,  are those from Subcases 2, 4, and 5. Let us study what happens if we try to add simultaneously arcs from
\begin{itemize}
    \item Subcases 2 and 4,
    \item Subcases 2 and 5,
    \item Subcases 4 and 5.
\end{itemize}

{Firstly} we will see that there cannot be arcs obtained from Subcases 2 and 4. Consider an arc $({v_b},y)$ with $y \in V(\vec{P}_{c+2})$ and an arc $(x',y')$ with $x',y'$ in $V(\vec{P}_{b+2})$ with $w<y'<x'<v$, then, considering the strongly connected digraph $D'$ generated by vertices in $D$ different from $v$ (see left of Figure \ref{grafos_39}), we obtain a digraph different from a cycle and from $D$ itself, and this contradicts Theorem~\ref{caracterizacion}.

Secondly, we will see that there cannot be arcs obtained from Subcases 2 and 5. Consider an arc $(x,y)$ with $x,y$ in $V(\vec{P}_{b+2})$ with $w<y<x<v$ and an arc $(x',w_b)$ with $x'$ in $V(\vec{P}_{c+2})$.
Considering the strongly connected digraph $D'$ generated by vertices in $D$ different from $w$ (see right of Figure \ref{grafos_39}) we obtain a digraph different from a cycle and from $D$ itself, which contradicts Theorem~\ref{caracterizacion}.
\begin{figure}[h!]
\centering
\includegraphics[scale=0.1]{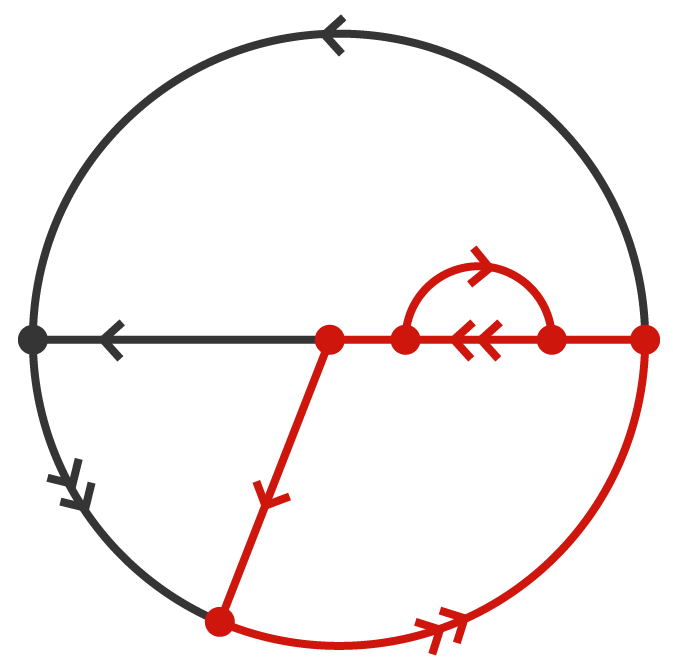}
\hspace{.5cm}
\includegraphics[scale=0.1]{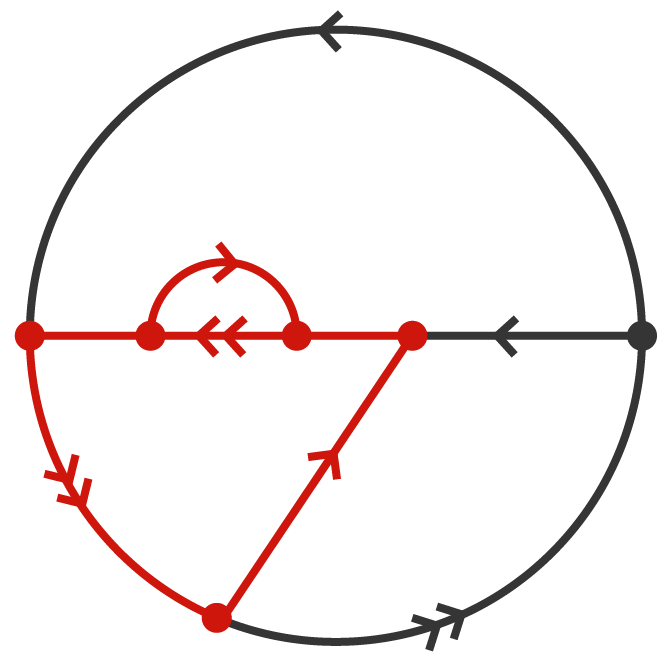}
\caption{Added arcs $(v_b,y)$ and $(x',y')$ with $y'<x'$ (left). Added arcs $(x,w_b)$ and $(x',y')$ with $y'<x'$ (right).}
\label{grafos_39}
\end{figure}

Finally, we analyze what happen if we add arcs obtained in Subcases 4 and 5 simultaneously. Suppose that we have an arc $({v}_b,y)$ with $y \in  V(\vec{P}_{c+2})$ and an arc $(x,{w}_b)$ with $x \in V(\vec{P}_{c+2})$. Assuming the natural order in $\vec{P}_{c+2}$ (from left to right), if $y\leq x$, we consider the strongly connected digraph $D'$ generated by removing $v$ (see left of Figure \ref{grafos_41}) {we obtain an induced strongly connected proper subdigraph different from a cycle and from an isolated vertex which contradicts Theorem~\ref{caracterizacion}.}\\
If $x<y$ and there exists $z \in  V(\vec{P}_{c+2})$ such that $x < z < y$. We then consider the strongly connected digraph $D'$ generated by removing $z$ (see right of Figure \ref{grafos_41}) {we obtain an induced strongly connected proper subdigraph different from a cycle and from an isolated vertex which contradicts Theorem~\ref{caracterizacion}.}

\begin{figure}[h!]
\centering
\includegraphics[scale=0.14]{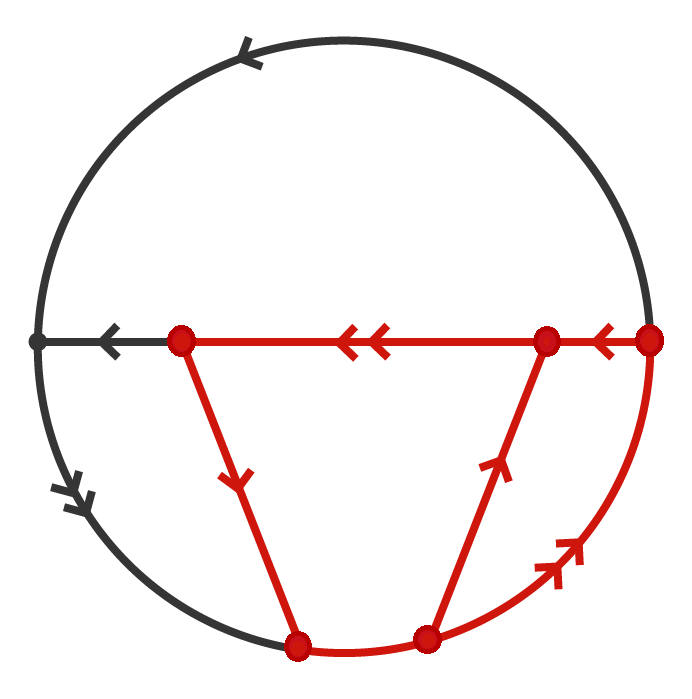}
\hspace{.5cm}
\includegraphics[scale=0.11]{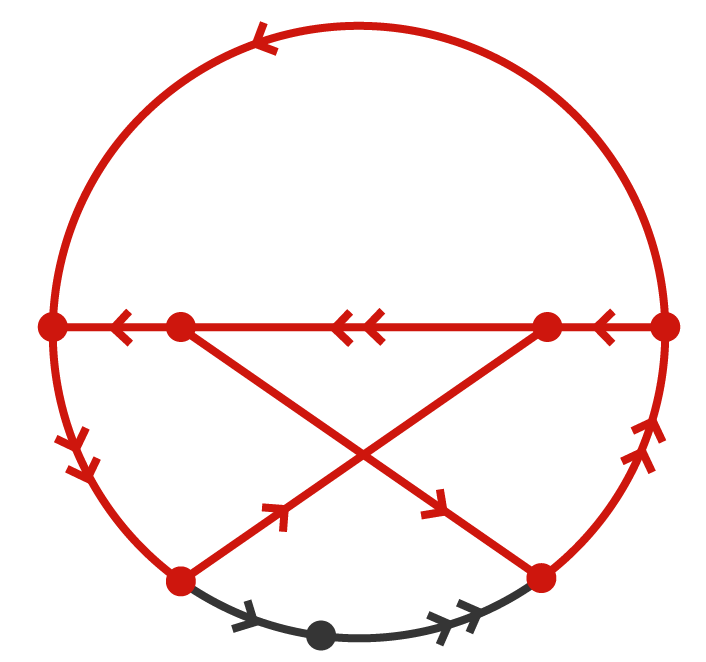}
\caption{Added arcs $(v_a,y)$ and $(x,w_b)$ with $y\leq x$ (left). $(v_a,y)$ and $(x,w_b)$ with $x<y$ (right).}
\label{grafos_41}
\end{figure}

{If $x<y$ and it does not exist such $z$, we have that $y=suc_{\vec{P}_c}(x)$ and we obtain a Type 5 digraph as shown in Figure \ref{grafos_43}.}

\begin{figure}[h!]
\centering
\includegraphics[scale=0.15]{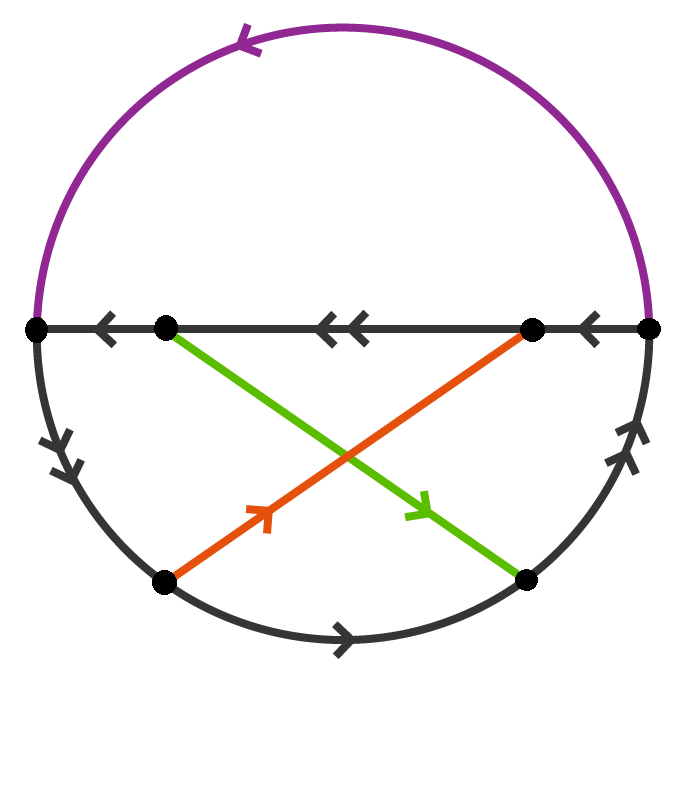}
\includegraphics[scale=0.15]{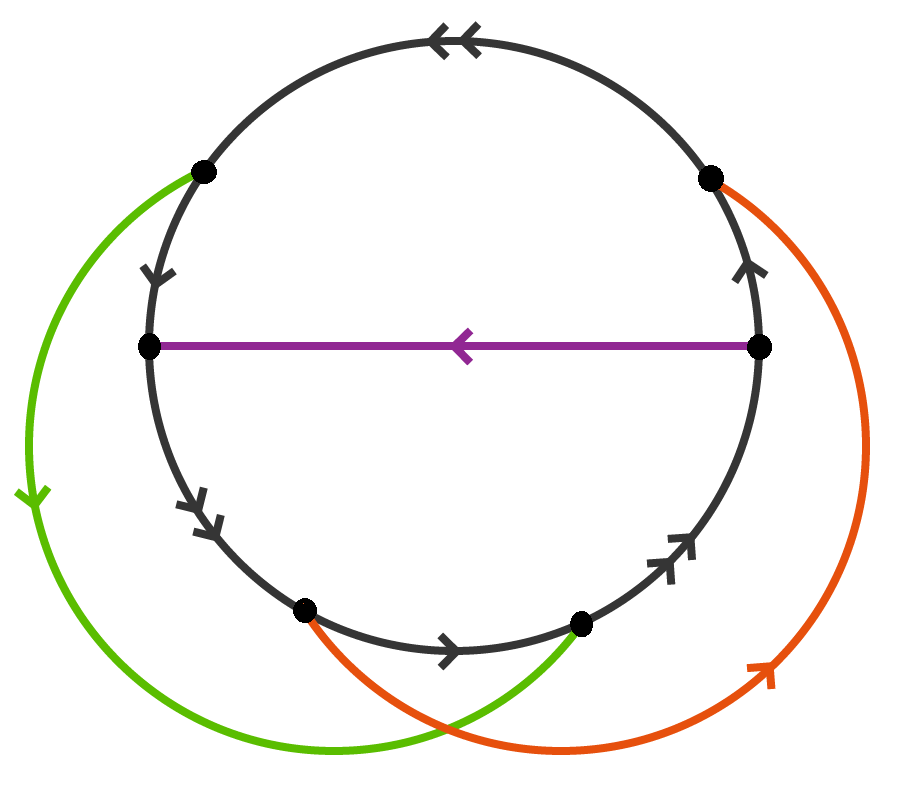}
\caption{Added arcs $(v_a,y)$ and $(x,w_b)$ with $y=suc_{\vec{P}_c}(x)$ and its convertion into a type $5$ digraph.}
\label{grafos_43}
\end{figure}

This concludes the proof.
\end{proof}

\begin{cor}
Let $D$ be a digraph in {$\mathcal{SCD}_3$} with $n$ vertices. If $D$ contains a $\theta (a,b,c)$-subdigraph, then $n=a+b+c+2$.
\end{cor}

\subsection{Resulting characterization of $\mathcal{SCD}_3$}\label{sec:teorema_todo_junto}

In this subsection we describe all digraphs in $\mathcal{SCD}_3$ using our previous results.
\begin{teo}
\label{main_theorem}
Let $D$ be a digraph in $\mathcal{SCD}_3$. Then $D$ belongs either to the $\infty$-Family or to the $\theta$-Family, i.e., $\mathcal{SCD}_3 = F_{\infty} \cup F_{\theta}.$ Precisely $D$ is either an $\infty$-digraph, or $\theta$-digraph, or a Type $1$, or a Type $2$ , or a Type $3$, or a Type $4$, or Type $5$ digraph.

\end{teo}

\begin{proof}
Since $D$ belongs to $\mathcal{SCD}_3$, we know that $D$ is a strongly connected digraph different from a cycle {and an isolated vertex}, and therefore, by Proposition~\ref{subdigrafosInftyTheta} we have that $D$ has an $\infty$-subdigraph or a $\theta$-subdigraph.

If $D$ has an $\infty$-subdigraph we know by Theorem~\ref{subdigrafoocho} that $D$ belongs to {the $\infty$-Family, e.g. it is either} an $\infty$-digraph, a Type $1$ or Type $2$ digraph. {If $D$ does not have an $\infty$-subdigraph, then}  the digraph $D$ verifies the hypothesis of Theorem~\ref{principal}, and then $D$ belongs to {the $\theta$-Family, e. g. $D$ is either}  a $\theta$-digraph, or a Type $3$, or a Type $4$, or Type $5$ digraph.
\end{proof}
In view of Proposition~\ref{prop:strong}, we have the following structural characterization of digraphs having 3 complementarity eigenvalues.
{\begin{cor}
Let $D$ be a digraph with three complementarity eigenvalues. Then, its strongly connected components are either isolated vertices, cycles, or any of the seven types of digraphs presented above. Moreover, at least one of these seven types must appear, and if two or more strongly connected components belongs to these seven types, then they should share the same spectral radius.
\end{cor}}

\section{Conclusions}
Throughout this work we presented a {structural} characterization result, and a {precise determination} of $\mathcal{SCD}_3$. Namely, the characterization presented in Theorem~\ref{caracterizacion} states that digraphs in $\mathcal{SCD}_3$ are those such that, when removing any vertex, the resulting digraph only has cycles or isolated vertices as strongly connected induced subdigraphs.
On the other hand, Theorem \ref{main_theorem}, which is the main contribution of this paper, states that any digraph in $\mathcal{SCD}_3$ belongs to one of the seven families of digraphs presented in the text, giving a complete description, determining exactly which digraphs have three complementarity eigenvalues.

One way to take these results one step further is to consider digraphs in $\mathcal{SCD}_4$. However, even a structural characterization result similar to the one given by Theorem~\ref{caracterizacion} , seems to be significantly harder, due to the combinatorial nature of the problem and the many possibilities that arise in the process. The key to characterize  digraphs with three complementarity eigenvalues was that strongly connected digraphs with two complementarity eigenvalues are cycles or isolated vertices. When considering digraphs $D$ in $\mathcal{SCD}_4$, the digraphs that result when removing a vertex have much more diversity, since there are seven families of such digraphs with three complementarity eigenvalues, and a combinatorial approach to analyze the cases which they can be modified to yield four complementarity eigenvalues may be unfeasible.

Another interesting and natural question is the following. Once digraphs in $\mathcal{SCD}_3$ {are characterized}, can one determine the digraph through its complementarity spectrum? The answer to this question, given in \cite{flor} is negative, since examples of non-isomorphic digraphs with the same {three} complementarity spectrum are presented. In the notation of this manuscript, those digraphs are Type $4$ digraphs. We can ask whether any family of digraphs with three complementarity eigenvalues is determined by their complementarity spectrum. As a motivating example, it follows from the discussion in {Remark~\ref{rem:dcs}} of Section \ref{sec:digrafos_con_infinito} that {some} digraphs in the $\infty$-Family can be distinguished by their complementarity spectrum.

\backmatter

\bmhead{Acknowledgments}
This research is part of the doctoral studies of F. Cubr\'{\i}a. V. Trevisan  acknowledges partial support of CNPq grants 409746/2016-9 and 310827/2020-5, and FAPERGS grant PqG 17/2551-0001.
M. Fiori and D. Bravo acknowledge the ﬁnancial support provided by ANII, Uruguay. F. Cubría thanks the doctoral scholarship from CAP-UdelaR. We thank Lucía Riera for the help with the figures.

\bibliography{biblio}


\begin{thebibliography}{16}
\ifx \bisbn   \undefined \def \bisbn  #1{ISBN #1}\fi
\ifx \binits  \undefined \def \binits#1{#1}\fi
\ifx \bauthor  \undefined \def \bauthor#1{#1}\fi
\ifx \batitle  \undefined \def \batitle#1{#1}\fi
\ifx \bjtitle  \undefined \def \bjtitle#1{#1}\fi
\ifx \bvolume  \undefined \def \bvolume#1{\textbf{#1}}\fi
\ifx \byear  \undefined \def \byear#1{#1}\fi
\ifx \bissue  \undefined \def \bissue#1{#1}\fi
\ifx \bfpage  \undefined \def \bfpage#1{#1}\fi
\ifx \blpage  \undefined \def \blpage #1{#1}\fi
\ifx \burl  \undefined \def \burl#1{\textsf{#1}}\fi
\ifx \doiurl  \undefined \def \doiurl#1{\url{https://doi.org/#1}}\fi
\ifx \betal  \undefined \def \betal{\textit{et al.}}\fi
\ifx \binstitute  \undefined \def \binstitute#1{#1}\fi
\ifx \binstitutionaled  \undefined \def \binstitutionaled#1{#1}\fi
\ifx \bctitle  \undefined \def \bctitle#1{#1}\fi
\ifx \beditor  \undefined \def \beditor#1{#1}\fi
\ifx \bpublisher  \undefined \def \bpublisher#1{#1}\fi
\ifx \bbtitle  \undefined \def \bbtitle#1{#1}\fi
\ifx \bedition  \undefined \def \bedition#1{#1}\fi
\ifx \bseriesno  \undefined \def \bseriesno#1{#1}\fi
\ifx \blocation  \undefined \def \blocation#1{#1}\fi
\ifx \bsertitle  \undefined \def \bsertitle#1{#1}\fi
\ifx \bsnm \undefined \def \bsnm#1{#1}\fi
\ifx \bsuffix \undefined \def \bsuffix#1{#1}\fi
\ifx \bparticle \undefined \def \bparticle#1{#1}\fi
\ifx \barticle \undefined \def \barticle#1{#1}\fi
\bibcommenthead
\ifx \bconfdate \undefined \def \bconfdate #1{#1}\fi
\ifx \botherref \undefined \def \botherref #1{#1}\fi
\ifx \url \undefined \def \url#1{\textsf{#1}}\fi
\ifx \bchapter \undefined \def \bchapter#1{#1}\fi
\ifx \bbook \undefined \def \bbook#1{#1}\fi
\ifx \bcomment \undefined \def \bcomment#1{#1}\fi
\ifx \oauthor \undefined \def \oauthor#1{#1}\fi
\ifx \citeauthoryear \undefined \def \citeauthoryear#1{#1}\fi
\ifx \endbibitem  \undefined \def \endbibitem {}\fi
\ifx \bconflocation  \undefined \def \bconflocation#1{#1}\fi
\ifx \arxivurl  \undefined \def \arxivurl#1{\textsf{#1}}\fi
\csname PreBibitemsHook\endcsname

\bibitem{Sachs1964}
\begin{barticle}
\bauthor{\bsnm{Sachs}, \binits{H.}}:
\batitle{Beziehungen zwischen den in einem graphen enthaltenen kreisen und
  seinem charakteristischen polynom}.
\bjtitle{Publ. Math. Debrecen}
\bvolume{11}(\bissue{1}),
\bfpage{119}--\blpage{134}
(\byear{1964})
\end{barticle}
\endbibitem

\bibitem{fiedler73}
\begin{barticle}
\bauthor{\bsnm{Fiedler}, \binits{M.}}:
\batitle{Algebraic connectivity of graphs}.
\bjtitle{Czechoslovak Mathematical Journal}
\bvolume{23}(\bissue{2}),
\bfpage{298}--\blpage{305}
(\byear{1973})
\end{barticle}
\endbibitem

\bibitem{Collatz1957}
\begin{barticle}
\bauthor{\bsnm{Von~Collatz}, \binits{L.}},
\bauthor{\bsnm{Sinogowitz}, \binits{U.}}:
\batitle{Spektren endlicher grafen}.
\bjtitle{Abhandlungen aus dem Mathematischen Seminar der Universit{\"a}t
  Hamburg}
\bvolume{21}(\bissue{1}),
\bfpage{63}--\blpage{77}
(\byear{1957}).
\bcomment{Springer}
\end{barticle}
\endbibitem

\bibitem{Brouwer12}
\begin{bbook}
\bauthor{\bsnm{Brouwer}, \binits{A.E.}},
\bauthor{\bsnm{Haemers}, \binits{W.H.}}:
\bbtitle{Spectra of Graphs}.
\bpublisher{Springer},
\blocation{New York, NY}
(\byear{2012}).
\doiurl{10.1007/978-1-4614-1939-6}
\end{bbook}
\endbibitem

\bibitem{Doob}
\begin{barticle}
\bauthor{\bsnm{Doob}, \binits{M.}}:
\batitle{Graphs with a small number of distinct eigenvalues}.
\bjtitle{Annals of the New York Academy of Sciences}
\bvolume{175},
\bfpage{104}--\blpage{110}
(\byear{2007})
\end{barticle}
\endbibitem

\bibitem{Cvetkovic1998}
\begin{bbook}
\bauthor{\bsnm{Cvetkovi{\'c}}, \binits{D.}},
\bauthor{\bsnm{Doob}, \binits{M.}},
\bauthor{\bsnm{Sachs}, \binits{H.}}:
\bbtitle{Spectra of Graphs: Theory and Applications}.
\bpublisher{Wiley},
\blocation{New York}
(\byear{1998})
\end{bbook}
\endbibitem

\bibitem{Seeger99}
\begin{barticle}
\bauthor{\bsnm{Seeger}, \binits{A.}}:
\batitle{Eigenvalue analysis of equilibrium processes defined by linear
  complementarity conditions}.
\bjtitle{Linear Algebra and its Applications}
\bvolume{292}(\bissue{1-3}),
\bfpage{1}--\blpage{14}
(\byear{1999})
\end{barticle}
\endbibitem

\bibitem{Fernandes2017}
\begin{barticle}
\bauthor{\bsnm{Fernandes}, \binits{R.}},
\bauthor{\bsnm{Judice}, \binits{J.}},
\bauthor{\bsnm{Trevisan}, \binits{V.}}:
\batitle{Complementary eigenvalues of graphs}.
\bjtitle{Linear Algebra and its Applications}
\bvolume{527},
\bfpage{216}--\blpage{231}
(\byear{2017})
\end{barticle}
\endbibitem

\bibitem{Seeger2018}
\begin{barticle}
\bauthor{\bsnm{Seeger}, \binits{A.}}:
\batitle{Complementarity eigenvalue analysis of connected graphs}.
\bjtitle{Linear Algebra and its Applications}
\bvolume{543},
\bfpage{205}--\blpage{225}
(\byear{2018})
\end{barticle}
\endbibitem

\bibitem{flor}
\begin{barticle}
\bauthor{\bsnm{Bravo}, \binits{D.}},
\bauthor{\bsnm{Cubría}, \binits{F.}},
\bauthor{\bsnm{Fiori}, \binits{M.}},
\bauthor{\bsnm{Trevisan}, \binits{V.}}:
\batitle{Complementarity spectrum of digraphs}.
\bjtitle{Linear Algebra and its Applications}
\bvolume{627},
\bfpage{24}--\blpage{40}
(\byear{2021})
\end{barticle}
\endbibitem

\bibitem{Olivieri}
\begin{barticle}
\bauthor{\bsnm{Olivieri}, \binits{A.}},
\bauthor{\bsnm{Rada}, \binits{J.}},
\bauthor{\bsnm{Rios~Rodriguez}, \binits{A.}}:
\batitle{Digraphs with few eigenvalues}.
\bjtitle{Utilitas Mathematica}
\bvolume{96},
\bfpage{89}--\blpage{99}
(\byear{2015})
\end{barticle}
\endbibitem

\bibitem{Adly2015}
\begin{barticle}
\bauthor{\bsnm{Adly}, \binits{S.}},
\bauthor{\bsnm{Rammal}, \binits{H.}}:
\batitle{A new method for solving second-order cone eigenvalue complementarity
  problems}.
\bjtitle{Journal of Optimization Theory and Applications}
\bvolume{165}(\bissue{2}),
\bfpage{563}--\blpage{585}
(\byear{2015})
\end{barticle}
\endbibitem

\bibitem{Facchinei2007}
\begin{bbook}
\bauthor{\bsnm{Facchinei}, \binits{F.}},
\bauthor{\bsnm{Pang}, \binits{J.-S.}}:
\bbtitle{Finite-dimensional Variational Inequalities and Complementarity
  Problems}.
\bpublisher{Springer},
\blocation{New York}
(\byear{2007})
\end{bbook}
\endbibitem

\bibitem{Pinto2008}
\begin{barticle}
\bauthor{\bparticle{Pinto~da} \bsnm{Costa}, \binits{A.}},
\bauthor{\bsnm{Seeger}, \binits{A.}}:
\batitle{Cone-constrained eigenvalue problems: theory and algorithms}.
\bjtitle{Computational Optimization and Applications}
\bvolume{45},
\bfpage{25}--\blpage{57}
(\byear{2010})
\end{barticle}
\endbibitem

\bibitem{Pinto2004}
\begin{barticle}
\bauthor{\bparticle{Pinto~da} \bsnm{Costa}, \binits{A.}},
\bauthor{\bsnm{Martins}, \binits{J.}},
\bauthor{\bsnm{Figueiredo}, \binits{I.}},
\bauthor{\bsnm{J{\'u}dice}, \binits{J.}}:
\batitle{The directional instability problem in systems with frictional
  contacts}.
\bjtitle{Computer Methods in Applied Mechanics and Engineering}
\bvolume{193}(\bissue{3-5}),
\bfpage{357}--\blpage{384}
(\byear{2004})
\end{barticle}
\endbibitem

\bibitem{Lin2012}
\begin{barticle}
\bauthor{\bsnm{Lin}, \binits{H.}},
\bauthor{\bsnm{Shu}, \binits{J.}}:
\batitle{A note on the spectral characterization of strongly connected bicyclic
  digraphs}.
\bjtitle{Linear Algebra and its Applications}
\bvolume{436}(\bissue{7}),
\bfpage{2524}--\blpage{2530}
(\byear{2012})
\end{barticle}
\endbibitem

\end{thebibliography}

\end{document}